\documentclass[12pt,reqno]{article}

\usepackage{amsfonts}
\usepackage{mathrsfs}

\usepackage{amsmath,amsthm,amsfonts,amssymb,latexsym,mathrsfs,color}

\usepackage[colorlinks=true,
linkcolor=webblue, filecolor=webbrown, citecolor=webred ]{hyperref}
\definecolor{webblue}{rgb}{0,.5,0}
\definecolor{webred}{rgb}{0,.5,0}
\definecolor{webbrown}{rgb}{.6,0,0}

\newtheorem{thm}{Theorem}[section]
\newtheorem{lem}[thm]{Lemma}

\newtheorem{prop}[thm]{Proposition}

\newtheorem{cl}{Claim}
\newcommand{\J}{\rm \textbf{JCF}}

\newcommand{\s}{\rm \textbf{SCF}}

\theoremstyle{definition}

\newtheorem{ex}[thm]{Example}
\newtheorem{ques}[thm]{Question}
\newtheorem{rem}[thm]{Remark}

\numberwithin{equation}{section}

\newcommand{\D}{\displaystyle}
\newcommand{\DF}[2]{\frac{\D#1}{\D#2}}

\title{Stieltjes
moment properties and continued fractions from combinatorial
triangles
\thanks{Supported partially by the National Natural
Science Foundation of China (Nos. 11971206, 12022105), the Natural
Science Fund for Distinguished Young Scholars of Jiangsu Province
(No. BK20200048) and the Young Talents Invitation Program of
Shandong Province.
\newline\hspace*{5mm}
   {\it Email address:} bxzhu@jsnu.edu.cn (B.-X. Zhu)}}
\author{Bao-Xuan Zhu}
\date{\footnotesize School of Mathematics and Statistics,
         Jiangsu Normal University,
         Xuzhou 221116, PR China}

\begin{document}

\maketitle

\begin{abstract}
Many combinatorial numbers can be placed in the following
generalized triangular array $[T_{n,k}]_{n,k\ge 0}$ satisfying the
recurrence relation:
\begin{equation*}
T_{n,k}=\lambda(a_0n+a_1k+a_2)T_{n-1,k}+(b_0n+b_1k+b_2)T_{n-1,k-1}+\frac{d(da_1-b_1)}{\lambda}(n-k+1)T_{n-1,k-2}
\end{equation*}
with $T_{0,0}=1$ and $T_{n,k}=0$ unless $0\le k\le n$ for suitable
$a_0,a_1,a_2,b_0,b_1,b_2,d$ and $\lambda$. For $n\geq0$, denote by
$T_n(q)$ the generating function of the $n$-th row. In this paper,
we develop various criteria for $\textbf{x}$-Stieltjes moment
property and $3$-$\textbf{x}$-log-convexity of $T_n(q)$ based on the
Jacobi continued fraction expression of $\sum_{n\geq0}T_n(q)t^n$,
where $\textbf{x}$ is a set of indeterminates consisting of $q$ and
those parameters occurring in the recurrence relation. With the help
of a criterion of Wang and Zhu [Adv. in Appl. Math. (2016)], we show
that the corresponding linear transformation of $T_{n,k}$ preserves
Stieltjes moment properties of sequences. Finally, we present some
related examples including factorial numbers, Whitney numbers,
Stirling permutations, minimax trees and peak statistics.
\end{abstract}
\newpage
 \tableofcontents

\section{Introduction}
\subsection{Total positivity and Stieltjes moment sequences}
Total positivity of matrices is an important and powerful concept
that arises often in various branches of mathematics, such as
classical analysis \cite{Sc30}, representation theory
\cite{Lus94,Rie03}, network analysis \cite{Pos06}, cluster algebras
\cite{BFZ96,FZ99}, combinatorics \cite{Bre95,GV85}, positive
Grassmannians and integrable systems \cite{KW14}. We refer the
reader to the monograph \cite{Kar68} for more details about total
positivity.

Let $A=[a_{n,k}]_{n,k\ge 0}$ be a matrix of real numbers. It is
called {\it totally positive} ({\it TP} for short) if all its minors
are nonnegative. It is called {\it TP$_r$} if all minors of order
$k\le r$ are nonnegative. For a sequence, the total positivity of
its Hankel matrix plays an important role in different fields. Let
us recall the definition. Given a sequence $\alpha=(a_k)_{k\ge 0}$,
define its {\it Hankel matrix} $H(\alpha)$ by
$$H(\alpha)=[a_{i+j}]_{i,j\ge 0}=
\left[
  \begin{array}{ccccc}
    a_0 & a_1 & a_2 & a_3 & \cdots \\
    a_1 & a_2 & a_3 & a_4 & \\
    a_2 & a_3 & a_4 & a_5 & \\
    a_3 & a_4 & a_5 & a_6 & \\
    \vdots &  &  &  & \ddots \\
  \end{array}
\right].$$

We say that $\alpha$ is a {\it Stieltjes moment} ({\it SM} for
short) sequence if it has the form
\begin{equation}\label{i-e}
a_k=\int_0^{+\infty}x^kd\mu(x),
\end{equation}
where $\mu$ is a non-negative measure on $[0,+\infty)$ (see
\cite[Theorem 4.4]{Pin10} for instance). Stieltjes proved that
$\alpha$ is an SM sequence if and only if  there exist nonnegative
numbers $\alpha_0,\alpha_1,\ldots$ such that
$$\sum_{n\geq0}a_nz^n=\frac{1}{1-\cfrac{\alpha_0z}{1-\cfrac{\alpha_1z}{1-\cdots}}}$$
in the sense of formal power series. It is well known that $\alpha$
is a Stieltjes moment sequence if and only if its Hankel matrix
$H(\alpha)$ is TP. Stieltjes moment problem is one of classical
moment problems and arises naturally in many branches of mathematics
\cite{ST43,Wid41}. Indeed, SM sequences are closely related to
log-convexity of sequences. The sequence $\alpha$ is called {\it
log-convex} if $a_{k-1}a_{k+1}\ge a_k^2$ for all $k\ge 1$. Clearly,
a sequence of positive numbers is log-convex if and only if its
Hankel matrix is TP$_2$. As a result, SM property implies
log-convexity. In addition, many log-convex sequences in
combinatorics have SM property. We refer the reader to Liu and Wang
\cite{LW07} and Zhu \cite{Zhu13} for log-convexity and Wang and Zhu
\cite{WZ16} and Zhu\cite{Zhu19,Zhu191} for SM property.

In what follows, concepts for log-convexity, SM property and total
positivity will be strengthened in a natural manner.

Let $\mathbb{P}$ denote the set of all positive integers and
$\mathbb{N}=\mathbb{P}\cup\{0\}$. Let $\textbf{x}=\{x_i\}_{i\in{I}}$
be a set of indeterminates. A matrix $M$ with entries in
$\mathbb{R}[\textbf{x}]$ is called \textbf{x-totally positive}
(\textbf{x}-TP for short) if all its minors are polynomials with
nonnegative coefficients in the indeterminates $\textbf{x}$ and is
called \textbf{x-totally positive of order $r$} ( \textbf{x}-TP$_r$
for short)if all its minors of order $k\le r$ are polynomials with
nonnegative coefficients in the indeterminates $\textbf{x}$. A
sequence $(\alpha_n(\textbf{x}))_{n\geq0}$ with values in
$\mathbb{R}[\textbf{x}]$ is called an \textbf{x-Stieltjes moment }
($\textbf{x}$-SM for short) sequence if its associated infinite
Hankel matrix is $\textbf{x}$-totally positive. We use
$f(x)\geq_{\textbf{x}}0$ to represent that all coefficients of the
polynomial $f(\textbf{x})$ are nonnegative. It is called
\textbf{x-log-convex } ($\textbf{x}$-LCX for short) if
$$\alpha_{n+1}(\textbf{x})\alpha_{n-1}(\textbf{x})-\alpha_n(\textbf{x})^2\geq_{\textbf{x}}0$$
for all $n\in \mathbb{P}$ and is called \textbf{strongly
x-log-convex } ($\textbf{x}$-SLCX for short) if
$$\alpha_{n+1}(\textbf{x})\alpha_{m-1}(\textbf{x})-\alpha_n(\textbf{x})\alpha_m(\textbf{x})\geq_{\textbf{x}}0$$
for all $n\geq m\geq1$. Clearly, an $\textbf{x}$-SM sequence is both
$\textbf{x}$-SLCX and $\textbf{x}$-LCX. Define an operator $\mathcal
{L}$ by
$$\mathcal {L}[\alpha_i(\textbf{x})]:=\alpha_{i-1}(\textbf{x})\alpha_{i+1}(\textbf{x})-\alpha_i(\textbf{x})^2$$
for $i\in \mathbb{P}$. Then the $\textbf{x}$-log-convexity of
$(\alpha_i(\textbf{x}))_{i\geq 0}$ is equivalent to $\mathcal
{L}[\alpha_i(\textbf{x})]\geq_{\textbf{x}}0$ for all $i\in
\mathbb{P}$. In general, we say that $(\alpha_i(\textbf{x}))_{i\geq
0}$ is {\it $\textbf{k}$-\textbf{x-log-convex}} if the coefficients
of $\mathcal {L}^m[\alpha_i(\textbf{x})]$ are nonnegative for all
$m\leq k$, where $\mathcal {L}^m=\mathcal {L}(\mathcal {L}^{m-1})$.
It is called {\it \textbf{infinitely x-log-convex}}  if
$(\alpha_i(\textbf{x}))_{i\geq 0}$ is $k$-$\textbf{x}$-log-convex
for every $k \in \mathbb{N}$.

If $\textbf{x}$ contains a unique indeterminate $q$, then they
reduce to $q$-LCX \cite{LW07}, $q$-SLCX
\cite{CTWY10,CWY11,Zhu13,Zhu14,Zhu182} and $q$-SM
\cite{WZ16,Zhu19,Zhu20}, respectively.

 Finally, for brevity, let
$\s[\alpha_{2i},\alpha_{2i+1};z]_{i\geq0}$ denote the {\it Stieltjes
continued fraction} expansion
$$\frac{1}{1-\cfrac{\alpha_0z}{1-\cfrac{\alpha_1z}{1-\cdots}}}.$$
The Stieltjes continued fraction expansion is closely related to the
{\it Jacobi continued fraction} expansion, denoted by
$\J[s_i,r_{i+1};z]_{i\geq0}$,
$$\frac{1}{1-s_0z-\cfrac{r_1z^2}{1-s_1z-\cfrac{r_2z^2}{1-\cdots}}}$$
by the famous contraction formulae
\begin{eqnarray}
\DF{1}{1-\DF{t_1z}{1-\DF{t_2z}{1-\ldots}}} &=&\DF{1}{1-
t_1z-\DF{t_1t_2z^2}{1-
(t_2+t_3)z-\DF{t_3t_4z^2}{1- (t_4+t_5)z-\ldots}}}\label{contraction}\\
&=&1+\DF{t_1z}{1- (t_1+t_2)z-\DF{t_2t_3z^2}{1-
(t_3+t_4)z-\DF{t_4t_5z^2}{1-\ldots}}}.\label{contraction+decom}
\end{eqnarray}
Thus we can write
\begin{eqnarray}
\s[\alpha_{2n},\alpha_{2n+1};z]_{n\geq0}=\J[\alpha_{2n-1}+\alpha_{2n},\alpha_{2n}\alpha_{2n+1};z]_{n\geq0}.
\end{eqnarray}

We refer the reader to the monograph \cite{JT80} about continued
fractions.

\subsection{Motivations}
It is well-known that many classical combinatorial arrays satisfy
certain recurrence relations. The following are some
 examples:

\begin{ex}\label{basic-qSM}
\begin{itemize}
 \item [\rm (i)]
 $\binom{n}{k}=\binom{n-1}{k}+\binom{n-1}{k-1}$, where $\binom{n}{k}$ is the binomial
 coefficients;
  \item [\rm (ii)]
$\left[
  \begin{array}{ccccc}
    n \\
   k\\
  \end{array}
\right]=(n-1)\left[
  \begin{array}{ccccc}
    n-1 \\
   k\\
  \end{array}
\right]+\left[
  \begin{array}{ccccc}
    n-1 \\
   k-1\\
  \end{array}
\right],$ where $\left[
  \begin{array}{ccccc}
    n \\
   k\\
  \end{array}
\right]$ is the signless Stirling number counting the number of
permutations of $n$ elements which are the product of $k$ disjoint
cycles;
  \item [\rm (iii)]
  $\left\{
  \begin{array}{ccccc}
    n \\
   k\\
  \end{array}
\right\}=k\left\{
  \begin{array}{ccccc}
    n-1 \\
   k\\
  \end{array}
\right\}+\left\{
  \begin{array}{ccccc}
    n-1 \\
   k-1\\
  \end{array}
\right\},$ where $\left\{
  \begin{array}{ccccc}
    n \\
   k\\
  \end{array}
\right\}$ is the Stirling number of the second kind enumerating the
number of partitions of an $n$-element set consisting of $k$
disjoint nonempty blocks;
  \item [\rm (iv)]
$\left\langle
  \begin{array}{ccccc}
    n \\
   k\\
  \end{array}
\right\rangle=k\left\langle
  \begin{array}{ccccc}
    n-1 \\
   k\\
  \end{array}
\right\rangle+(n-k+1)\left\langle
  \begin{array}{ccccc}
    n-1 \\
   k-1\\
  \end{array}
\right\rangle$, where $\left\langle
  \begin{array}{ccccc}
    n \\
   k\\
  \end{array}
\right\rangle$ is the classical Eulerian number counting the number
of permutations of $n$ elements having $k-1$ descents;
\item [\rm (v)]
$\mathscr{B}_{n, k}=(k+1) \mathscr{B}_{n-1,k}+n
\mathscr{B}_{n-1,k-1}+(n-k+1)\mathscr{B}_{n-1,k-2},$ where
$\mathscr{B}_{n, k}$ is  the number of symmetric tableaux of size
$2n+1$ with $k+1$ diagonal cells \cite{ABN13};
\item [\rm (vi)]
$\mathscr{T}_{n,k}=(k+1)\mathscr{T}_{n-1,k}+(n+1)\mathscr{T}_{n-1,k-1}+(n-k+1)\mathscr{T}_{n-1,k-2},$
where $\mathscr{T}_{n,k}$ is the number of staircase tableaux of
size $n$ with $k$ labels $\alpha$ or $\delta$ in the diagonal
\cite{ABD13}.
\end{itemize}
\end{ex}

These examples can be placed in a common framework. Let $\mathbb{R}$
(resp. $\mathbb{R^{+}}$, $\mathbb{R^{\geq}}$) be the set of all
(resp., positive, nonnegative) real numbers. Let $\{a_1,b_1\}\in
\mathbb{R}$, $\lambda\in\mathbb{R^{+}}$ and
$\{a_0,a_2,b_0,b_2,d\}\subseteq \mathbb{R^{\geq}}$. Define a
generalized triangular array $[T_{n,k}]_{n,k\ge 0}$ by the
recurrence relation:
\begin{equation}\label{Recurece+TT}
T_{n,k}=\lambda(a_0n+a_1k+a_2)T_{n-1,k}+(b_0n+b_1k+b_2)T_{n-1,k-1}+\frac{d(da_1-b_1)}{\lambda}(n-k+1)T_{n-1,k-2}
\end{equation}
with $T_{0,0}=1$ and $T_{n,k}=0$ unless $0\le k\le n$. We also
denote its row-generating function by
$T_n(q)=\sum_{k\geq0}T_{n,k}q^k$ for $n\geq0$.

 Many positivity properties of $[T_{n,k}]_{n,k}$ have been derived
 for $d=0$, see Kurtz \cite{Kur72} for log-concavity of each
row sequence,
 Wang and Yeh~\cite{WYjcta05} for P\'olya Frequency property of each row
 sequence, Liu and Wang
\cite{LW07} for the $q$-log-convexity of $(T_n(q))_{n\geq0}$, Chen
{\it et al.}~\cite{CWY11} for the strong $q$-log-convexity of
$(T_n(q))_{n\geq0}$ and Zhu \cite{Zhu182} for linear transformations
of $T_{n,k}$ preserving the strong $q$-log-convexity. It was proved
that row-generating functions $T_n(q)$ for Stirling triangle of the
second kind and Eulerian triangle form a $q$-Stieltjes moment
sequence in $q$, respectively, see Wang and Zhu \cite{WZ16} for
instance. For $d\neq0$, recently in \cite{Zhu201}, we proved for
certain special case that $(T_n(q))_{n\geq0}$ is a $q$-Stieltjes
moment sequence.

The aim of this paper is to consider the $q$-Stieltjes moment
property of $(T_n(q))_{n\geq0}$ for the general case. In addition,
in view of (\ref{Recurece+TT}), clearly, all elements $T_{n,k}$ are
polynomials in the eight parameters $a_0,a_1,a_2,b_0,b_1$, $b_2$,
$d$ and $\lambda$, and $T_n(q)$ can be regarded as a polynomial in
nine indeterminates $a_0,a_1,a_2,b_0,$ $b_1$, $b_2$, $d$, $\lambda$
and $q$. It is natural to consider the following multi-variable
question.

\begin{ques}\label{Que}
Assume that the array $[T_{n,k}]_{n,k}$ is defined in
(\ref{Recurece+TT}). When is $(T_n(q))_{n\geq0}$ an
$\textbf{x}$-Stieltjes moment sequence with
$\textbf{x}=(a_0,a_1,a_2,b_0,b_1,b_2,d,\lambda,q)$ ?
\end{ques}

In order to answer Question \ref{Que}, our main tool is to use
continued fraction expressions in Sections $2$ and $3$. For the
array $[T_{n,k}]_{n,k}$ in (\ref{Recurece+TT}) without the term
$T_{n-1,k-2}$, Sokal also conjectured that $(T_n(q))_{n\geq0}$ is
$(a_0,a_1,a_2,b_0,b_1,b_2,q)$-Stieltjes moment \cite{SS,Sok}. We
prove the next result.

\begin{thm}\label{thm+Ring+PSE+SM}
If a triangle $[T_{n,k}]_{n,k}$ satisfies any of the following:
\begin{itemize}
 \item [\rm (i)]
  $T_{n,k}=[a_0(n-1)+a_2]T_{n-1,k}+[b_0(n-1)+b_2]T_{n-1,k-1};$
   \item [\rm (ii)] $T_{n,k}=[a_0(b_0+ b_1)(n-1)+a_0 b_1k+a_0
b_2]T_{n-1,k}+[b_0(n-1)+b_1(k-1)+b_2]T_{n-1,k-1}$;
  \item [\rm (iii)]
  $T_{n,k}=(a_1k+a_2)T_{n-1,k}+[b_1(k-1)+b_2]T_{n-1,k-1};$
  \item [\rm (iv)]  $T_{n,k}=[a_0(n-k-1)+a_2]T_{n-1,k}+[b_0(n-k)+b_2]T_{n-1,k-1};$
  \item [\rm (v)]
  $T_{n,k}=(a_1k+a_2)T_{n-1,k}+(b_0n-b_0k+b_2)T_{n-1,k-1}$ \text{for}
  $0\in\{a_2,b_2,a_1-a_2,b_0-b_2\}$;
  \item [\rm (vi)]
  $T_{n,k}=b_0\left(n-2k+\frac{2a_2-a_1}{a_1}\right)T_{n-1,k}+[a_1\,(n-k)+a_2]T_{n-1,k-1}$;
  \item [\rm (vii)]
  $T_{n,k}=(a_1\,k+a_2)T_{n-1,k}+b_0\left(2k-n+\frac{2a_2-a_1}{a_1}\right)T_{n-1,k-1}$,
\end{itemize}
where  $T_{0,0}=1$ and $T_{n,k}=0$ unless $0\le k\le n$, then
$(T_n(q))_{n\geq0}$ is an $\textbf{x}$-Stieltjes moment and
$3$-$\textbf{x}$-log-convex sequence with
$\textbf{x}=(a_0,a_1,a_2,b_0,b_1,b_2,q)$.
\end{thm}

Note that the next relationship was proved in \cite{Zhu20}.
\begin{thm}\label{thm+m+tran}
 Let $[T_{n,k}]_{n,k\ge 0}$ be defined
in (\ref{Recurece+TT}). Then there exists an array
$[A_{n,k}]_{n,k\geq0}$ satisfying the recurrence relation
 \begin{eqnarray*}
A_{n,k}
&=&[[b_0+d(a_1-a_0)]n+(b_1-2da_1)k+b_2+d(a_1-a_2)]A_{n-1,k-1}+\\
&&(a_0n+a_1k+a_2)A_{n-1,k}
\end{eqnarray*}
with $A_{0,0}=1$ and $A_{n,k}=0$ unless $0\le k\le n$ such that
their row-generating functions satisfy
\begin{eqnarray}\label{rel+T+A}
T_n(q)=(\lambda+dq)^nA_n(\frac{q}{\lambda+dq})
\end{eqnarray} for $n\geq0$.
\end{thm}
Thus, for the array $[T_{n,k}]_{n,k}$ satisfying a four term
recurrence relation in (\ref{Recurece+TT}), we can get the
corresponding results by Theorem \ref{thm+Ring+PSE+SM} and Theorem
\ref{thm+m+tran}. For instance, we list three concise cases as
follows.

\begin{thm}\label{thm+three+PSE+SM}
If a triangle $[T_{n,k}]_{n,k}$ satisfies any of the following:
\begin{itemize}
\item [\rm (i)]
  $T_{n,k}=\lambda(a_1k+a_2)T_{n-1,k}+[-da_1n+(b_1+2da_1)k+b_2-b_1-d(a_1-a_2)]T_{n-1,k-1}-\\\frac{d(da_1+b_1)}{\lambda}(n-k+1)T_{n-1,k-2};$
 \item [\rm (ii)]
  $T_{n,k}=\lambda(a_0n-a_0k+a_2-a_0)T_{n-1,k}+[(b_0+2da_0)(n-k)+b_2+da_2]T_{n-1,k-1}+\\
 \frac{d(b_0+da_0)}{\lambda}(n-k+1)T_{n-1,k-2};$
 \item [\rm (iii)]\footnote{This case was also proved in \cite{Zhu201}. We will give a different proof.}
  $T_{n,k}=\lambda(a_1k+a_2)T_{n-1,k}+[(b_0-da_1)n-(b_0-2da_1)k+b_2-d(a_1-a_2)]T_{n-1,k-1}+\\
\frac{d(b_0-da_1)}{\lambda}(n-k+1)T_{n-1,k-2},$
\end{itemize}
where  $T_{0,0}=1$ and $T_{n,k}=0$ unless $0\le k\le n$, then
$(T_n(q))_{n\geq0}$ is an $\textbf{x}$-Stieltjes moment and
$3$-$\textbf{x}$-log-convex sequence with
$\textbf{x}=(a_0,a_1,a_2,b_0,b_1,b_2,d,\lambda,q)$.
\end{thm}

 Let $M=[M_{n,k}]_{n,k\ge 0}$ be an infinite matrix.  For $n\geq0$, define the $M$-convolution
\begin{eqnarray}\label{a-c}
z_n=\sum_{k=0}^{n}M_{nk}x_ky_{n-k}.
\end{eqnarray}
We say that \eqref{a-c} preserves the Stieltjes moment property: if
both $(x_n)_{n\ge 0}$ and $(y_n)_{n\ge 0}$ are Stieltjes moment
sequences, then so is $(z_n)_{n\ge 0}$.

Using positive definiteness of the quadratic form, P\'olya and
Szeg\"o~\cite[Part VII, Theorem 42]{PS64} proved that the binomial
convolution
$$z_n=\sum_{k=0}^{n}\binom{n}{k}x_ky_{n-k}$$
preserves the Stieltjes moment property for real numbers. Recently,
more and more triangular convolutions preserving the Stieltjes
moment property for real numbers, see Wang and Zhu \cite{WZ16} and
Zhu \cite{Zhu201}. In addition, the next generalized result was
proved.

\begin{lem}\label{lem+conv}\emph{\cite{WZ16}}
Let $M_n(q)=\sum_{k=0}^{n}M_{n,k}q^k$ be the $n$-th row generating
function of a matrix $M$. Assume that $(M_n(q))_{n\ge 0}$ is a
Stieltjes moment sequence for any fixed $q\ge 0$. Then the
$M$-convolution \eqref{a-c} preserves the Stieltjes moment property
for real numbers.
\end{lem}
Combining Theorems \ref{thm+Ring+PSE+SM} and \ref{thm+three+PSE+SM}
and Lemma \ref{lem+conv}, we immediately have the next result.
\begin{thm}\label{thm+T+conv+SM}
Let $\{a_0,a_1,a_2,b_0,b_1,b_2,d\}\subseteq\mathbb{R^{\geq}}$ and
$\lambda>0$. If the triangular array $[T_{n,k}]_{n,k}$ satisfies any
recurrence relation in Theorem \ref{thm+Ring+PSE+SM} and Theorem
\ref{thm+three+PSE+SM}, then its triangle-convolution
$$z_n=\sum_{k=0}^{n}T_{n,k} x_ky_{n-k},\quad n=0,1,2,\ldots$$
preserves the Stieltjes moment property for real numbers.
\end{thm}

For the row-generating function $T_n(q)$ of $[T_{n,k}]_{n,k}$, we
also give a result for the $\textbf{x}$-Stieltjes moment property of
$T_n(q)$ by taking $q$ to be a fixed $\mu$.
\begin{thm}\label{thm+sequence+SM}
If an array $[T_{n,k}]_{n,k}$ satisfies the recurrence relation:
  $$T_{n,k}=(a_0n-\mu b_1k+a_2)T_{n-1,k}+(b_0n+b_1k+b_2)T_{n-1,k-1},$$
where $T_{n,k}=0$ unless $0\le k\le n$ and $T_{0,0}=1$, then
$(T_n(\mu))_{n\geq0}$ is an $\textbf{x}$-Stieltjes moment sequence
with $\textbf{x}=(a_0,a_2,b_0,b_1,b_2,\mu)$.
\end{thm}

\section{Total positivity and continued fractions}
In this section, we will present some criteria for total positivity
from combinatorial arrays and continued fraction expansions.
\begin{thm}\label{thm+Hakel+Jacbi+main}
Let $\{r_n(\textbf{x}),
s_n(\textbf{x}),t_n(\textbf{x})\}\subseteq\mathbb{R}[\textbf{x}]$
for $n\in \mathbb{N}$. Assume that an array $[D_{n,k}]_{n,k}$
satisfies the recurrence relation:
 \begin{equation}\label{recurr+PJS}
D_{n,k}=r_{k-1}(\textbf{x})D_{n-1,k-1}+s_k(\textbf{x})D_{n-1,k}+t_{k+1}(\textbf{x})D_{n-1,k+1},
 \end{equation}
were $D_{n,k}=0$ unless $0\le k\le n$ and $D_{0,0}=1$. Then we have
the following results:
\begin{itemize}
\item [\rm (i)]
The ordinary generating function has the Jacobi continued fraction
expression
\begin{eqnarray*}
\sum\limits_{n=0}^{\infty}D_{n,0} z^n=\DF{1}{1-
s_0(\textbf{x})z-\DF{r_0(\textbf{x})t_1(\textbf{x})z^2}{1-
s_1(\textbf{x})z-\DF{r_1(\textbf{x})t_2(\textbf{x})z^2}{1-
s_2(\textbf{x})z-\ldots}}}.
\end{eqnarray*}
\item [\rm (ii)]
If the tridiagonal matrix
\begin{equation*}\label{J-eq}
J(r(\textbf{x}),s(\textbf{x}),t(\textbf{x}))=\left[
\begin{array}{ccccc}
s_0(\textbf{x}) & r_0(\textbf{x}) &  &  &\\
t_1(\textbf{x}) & s_1(\textbf{x}) & r_1(\textbf{x}) &\\
 & t_2(\textbf{x}) & s_2(\textbf{x}) &r_2(\textbf{x}) &\\
& & \ddots&\ddots & \ddots \\
\end{array}\right]
\end{equation*}  is $\textbf{x}$-TP$_r$ (resp., $\textbf{x}$-TP), then the Hankel matrix
$[D_{i+j,0}]_{i,j\geq0}$ is $\textbf{x}$-TP$_r$ (resp.,
$\textbf{x}$-TP).
\item [\rm (iii)]
If $J(r(\textbf{x}),s(\textbf{x}),t(\textbf{x}))$ is
$\textbf{x}$-TP$_{k+1}$ for $1\leq k\leq 3$, then
$(D_{n,0})_{n\geq0}$ is $k$-$\textbf{x}$-log-convex.
\end{itemize}
\end{thm}
\begin{proof}
(i)
 Let $h_k(z)=\sum_{n\geq k}D_{n,k}z^n$ for
$k\geq0$. It follows from the recurrence relation (\ref{recurr+PJS})
that we have
\begin{eqnarray*}
h_0(z)&=&1+s_0zh_0(z)+t_1zh_1(z), \\
h_k(z)&=&r_{k-1}zh_{k-1}(z)+s_{k}zh_k(z)+t_{k+1}zh_{k+1}(z)
\end{eqnarray*}
for $k\geq1$, which imply
\begin{eqnarray*}
\frac{h_0(z)}{1}&=&\frac{1}{1-s_0z-t_1z\frac{h_1(z)}{h_0(z)}}, \\
\frac{h_1(z)}{h_0(z)}&=&\frac{r_0z}{1-s_1z-t_2z\frac{h_2(z)}{h_1(z)}},\\
&\vdots&\\
\frac{h_k(z)}{h_{k-1}(z)}&=&\frac{r_{k-1}z}{1-s_{k}z-t_{k+1}z\frac{h_{k+1}(z)}{h_k(z)}}.
\end{eqnarray*}
Thus we get
\begin{eqnarray*} \sum\limits_{n=0}^{\infty}D_{n,0}
z^n=h_0(z)=\DF{1}{1- s_0z-\DF{r_0t_1z^2}{1- s_1z-\DF{r_1t_2z^2}{1-
s_2z-\ldots}}}.
\end{eqnarray*}

(ii) Let $\alpha=(D_{n,0})_{n\geq0}$ and $D=[D_{n,k}]_{n,k\geq0}$.
We will show the following fundamental expression.
\begin{cl}\label{Cl}
We have the fundamental formula
$$DVD^T=H(\alpha),$$
where
\[V=\left[\begin{array}{ccccc}
V_{0}&&&&\\
&V_{1}&&&\\
&&V_{2}&&\\
&&&\ddots\\
\end{array}\right]\]
with $V_0=1,V_n=\prod_{i=1}^{n}t_ir_{i-1}^{-1}$ for $n\geq1$.
\end{cl}
\begin{proof}
 In order to show $DVD^T=H(\alpha)$, it
suffices to prove that
\begin{eqnarray}
\sum_{k\geq0}D_{n,k}D_{m,k}V_k=D_{n+m,0}
\end{eqnarray} for any nonnegative
integers $n$ and $m$. It is obvious that
$$\sum_{k\geq0}D_{n,k}D_{0,k}V_k=D_{0,0}D_{n,0}V_0=D_{n,0}.$$ Assume
that the assertion is true for all $i\leq m-1$ and all $n$. Then
\begin{eqnarray*}
\sum_{k\geq0}D_{n,k}D_{m,k}V_k
&=&\sum_{k\geq0}D_{n,k}\left[r_{k-1}D_{m-1,k-1}+s_kD_{m-1,k}+t_{k+1}D_{m-1,k+1}\right]V_k\\
&=&\sum_{k\geq0}D_{m-1,k}\left[r_{k}V_{k+1}D_{n,k+1}+s_kV_kD_{n,k}+t_{k}V_{k-1}D_{n,k-1}\right]\\
&=&\sum_{k\geq0}D_{m-1,k}\left[t_{k+1}D_{n,k+1}+s_kD_{n,k}+r_{k-1}D_{n,k-1}\right]V_k\\
&=&\sum_{k\geq0}D_{m-1,k}D_{n+1,k}V_k\\
&=&D_{n+m,0},
\end{eqnarray*}
where the third equality is obtained by using
$r_kV_{k+1}=t_{k+1}V_k$.
\end{proof}

For the triangular array $[D_{n,k}]_{n,k}$, we can construct an
associated triangular array $[D^*_{n,k}]_{n,k}$ as follows:
\begin{equation}\label{rr+three}
D^*_{n,k}=D^*_{n-1,k-1}+s_k\,D^*_{n-1,k}+r_{k}t_{k+1}\,D^*_{n-1,k+1}
\end{equation}
with $D^*_{n,k}=0$ unless $0\le k\le n$ and $D^*_{0,0}=1$. Clearly,
$D^*_{n,0}=D_{n,0}$ because its ordinary generating function has the
same Jacobi continued fraction expansion by (i). Let $D^*$ denote
the matrix $[D^*_{n,k}]_{n,k\geq0}$. By the above Claim \ref{Cl}, we
immediately have
\begin{equation}\label{Hankel+decom}
D^*V^*{(D^*)}^T=H(\alpha),
\end{equation}
 where
\[V^*=\left[\begin{array}{ccccc}
V^*_{0}&&&&\\
&V^*_{1}&&&\\
&&V^*_{2}&&\\
&&&\ddots\\
\end{array}\right]\]
with $V^*_0=1,V^*_n=\prod_{i=1}^{n}t_ir_{i-1}$ for $n\geq1$.

Thus, applying the classical Cauchy-Binet formula to
(\ref{Hankel+decom}), in order to prove that $H(\alpha)$ is
$\textbf{x}$-TP$_r$ (resp., $\textbf{x}$-TP), it suffices to
demonstrate that the matrix $D^*$ is $\textbf{x}$-TP$_r$ (resp.,
$\textbf{x}$-TP). This follows from the next two claims.

\begin{cl}\label{cl-equal+tridilog}
The matrix
\[J=\left[\begin{array}{ccccc}
s_0&r_0&&&\\
t_1&s_1&r_1&&\\
&t_2&s_2&r_2&\\
&&\ddots&\ddots&\ddots
\end{array}\right]\] is $\textbf{x}$-TP$_r$ (resp., $\textbf{x}$-TP)
in $\mathbb{R}[\textbf{x}]$ if and only if the matrix
\[J^*=\left[\begin{array}{ccccc}
s_0&1&&&\\
r_0t_1&s_1&1&&\\
&r_1t_2&s_2&1&\\
&&\ddots&\ddots&\ddots
\end{array}\right]
\]
is $\textbf{x}$-TP$_r$ (resp., $\textbf{x}$-TP) in
$\mathbb{R}[\textbf{x}]$.
\end{cl}

\begin{proof}
It suffices to prove the corresponding result for their $n$-th
leading principal submatrices, see Theorem 4.3 of Pinkus
\cite{Pin10}. That is the matrix
\[J_n=\left[\begin{array}{ccccc}
s_0&r_0&&&\\
t_1&s_1&r_1&&\\
&t_2&s_2&r_2&\\
&&\ddots&\ddots&\ddots\\
&&&t_n&s_n\\
\end{array}\right]\] is $\textbf{x}$-TP$_r$ (resp., $\textbf{x}$-TP)
in $\mathbb{R}[\textbf{x}]$ if and only if the matrix
\[J^*_n=\left[\begin{array}{ccccc}
s_0&1&&&\\
r_0t_1&s_1&1&&\\
&r_1t_2&s_2&1&\\
&&\ddots&\ddots&\ddots\\
&&&r_{n-1}t_n&s_n\\
\end{array}\right]
\]
is $\textbf{x}$-TP$_r$ (resp., $\textbf{x}$-TP) in
$\mathbb{R}[\textbf{x}]$. This follows by induction on $n$ since
$\det{J_n}=\det{J^*_n}$ from the following identities
\begin{eqnarray*}
\det{J_n}&=&s_n\det{J_{n-1}}-r_{n-1}t_n\det{J_{n-2}},\\
\det{J^*_n}&=&s_n\det{J^*_{n-1}}-r_{n-1}t_n\det{J^*_{n-2}}.
\end{eqnarray*}
\end{proof}

\begin{cl}
If the matrix
\[J=\left[\begin{array}{ccccc}
s_0&r_0&&&\\
t_1&s_1&r_1&&\\
&t_2&s_2&r_2&\\
&&\ddots&\ddots&\ddots
\end{array}\right]\] is $\textbf{x}$-TP$_r$ (resp., $\textbf{x}$-TP)
in $\mathbb{R}[\textbf{x}]$, then so is the triangular matrix $D^*$.
\end{cl}
\begin{proof}
Let $\overline{D^*}$ denote the matrix obtained from $D^*$ by
deleting its first row. Also let $\overline{D_n^*}$ and $D_n^*$
denote the $(n+1)$-th leading principal submatrices of
$\overline{D^*}$ and $D^*$, respectively. In order to prove that the
triangular matrix $D^*$ is $\textbf{x}$-TP$_r$ (resp.,
$\textbf{x}$-TP), it suffices to prove that $D_n^*$ is
$\textbf{x}$-TP$_r$ (resp., $\textbf{x}$-TP). It follows from Claim
\ref{cl-equal+tridilog} that $J^*_n$ is $\textbf{x}$-TP$_r$ (resp.,
$\textbf{x}$-TP). By (\ref{rr+three}), we have
$\overline{D_n^*}=D^*_nJ^*_n$. By induction on $n$, we immediately
get that $D^*_n$ for $n\in N$ is $\textbf{x}$-TP$_r$ (resp.,
$\textbf{x}$-TP).
\end{proof}

(iii) For brevity, we write $D_{n,0}$ for $D_n$. It follows from
(ii) that the Hankel matrix $[D_{i+j}]_{i,j\geq0}$ is
$\textbf{x}$-TP$_{r+1}$ for $1\leq r\leq3$. We will use some
identities (see \cite{Zhu18} for instance):
\begin{eqnarray*}
\mathcal {L}(D_k)&=&\left|\begin{array}{cc}
D_{k-1}&D_{k}\\
D_{k}&D_{k+1}
\end{array}\right|,\\
\mathcal {L}^2(D_k)&=&\mathcal {L}(D_{k-1})\mathcal
{L}(D_{k+1})-\left[\mathcal {L}(D_k)\right]^2\nonumber\\
&=&\left(D_{k+2}D_{k}-D_{k+1}^2\right)\left(D_{k}D_{k-2}-D_{k-1}^2\right)-\left(D_{k+1}D_{k-1}-D_k^2\right)^2\nonumber\\
&=&D_k\left|\begin{array}{ccc}
D_{k-2}&D_{k-1}&D_{k}\\
D_{k-1}&D_k&D_{k+1}\\
D_k&D_{k+1}&D_{k+2}
\end{array}\right|,\\
\mathcal {L}^3(D_k)&=&\mathcal {L}^2(D_{k-1})\mathcal
{L}^2(D_{k+1})-\left[\mathcal {L}^2(D_k)\right]^2\\
&=&(D_{k+1}D_{k-1}-D^2_k) D^2_k\left|\begin{array}{cccc}
D_{k-3}&D_{k-2}&D_{k-1}&D_{k}\\
D_{k-2}&D_{k-1}&D_{k}&D_{k+1}\\
D_{k-1}&R_{k}&D_{k+1}&D_{k+2}\\
D_k&D_{k+1}&D_{k+2}&D_{k+3}
\end{array}\right|+\\
&& (D_{k+1}D_{k-1}-D^2_k)\left|\begin{array}{ccc}
D_{k-3}&D_{k-2}&D_{k-1}\\
D_{k-2}&D_{k-1}&D_{k}\\
D_{k-1}&D_{k}&D_{k+1}
\end{array}\right|\left|\begin{array}{ccc}
D_{k-1}&D_{k}&D_{k+1}\\
D_{k}&D_{k+1}&D_{k+2}\\
D_{k+1}&D_{k+2}&D_{k+3}
\end{array}\right|.
\end{eqnarray*}
Obviously, $(D_n)_{n\geq0}$ is strongly \textbf{x}-log-convex if
$[D_{i+j}]_{i,j\geq0}$ is \textbf{x}-TP$_{2}$. If
$[D_{i+j}]_{i,j\geq0}$ is \textbf{x}-TP$_{r+1}$ for $r=2,3$, then
$(D_n)_{n\geq0}$ is $k$-\textbf{x}-log-convex. We complete the
proof.
\end{proof}
\begin{rem}
When $r_n=1$ for $n\geq0$, the array $[D_{n,k}]_{n,k}$ for real
numbers was called the recursive triangle by Aigner \cite{Aig01}.
Aigner also gave a determinant method to get the corresponding
continued fraction and Flajolet \cite{Fla80} also presented a
combinatorial interpretation.
\end{rem}

\begin{rem}
We refer the reader to \cite{Zhu13,Zhu14} for the original idea of
total positivity in Theorem \ref{thm+Hakel+Jacbi+main}. Note that we
can not directly get Hankel-total positivity using $V$ because $V_n$
may be a rational function not a polynomial. Our proof indicates the
relation between arrays $[D_{n,k}]$ and $[D^*_{n,k}]$.
\end{rem}

\begin{rem}
Generally speaking, it is much easier to deal with total positivity
of the tridiagonal matrix $J$ than that of $J^*$. In fact, under
keeping the product of $\textbf{r}$ and $\textbf{t}$, we can choose
different pairs of $\textbf{r}$ and $\textbf{t}$.
\end{rem}

Recall the concept of $\gamma$-binomial transformation. Given a
sequence $(a_n)_{n\geq0}$, its $\gamma$-binomial transformation is
defined to be
\begin{eqnarray}
a^\circ_n=\sum_{k=0}^n\binom{n}{k}a_k\gamma^{n-k}
\end{eqnarray} for $n\geq0$. For $\gamma=1$, it
 reduces to the famous binomial transformation.
More generally, for an array $[A_{n,k}]_{n,k}$, define its {\it
$\gamma$-binomial transformation} $[A^\circ_{n,k}]_{n,k}$ by
\begin{eqnarray}
A^\circ_{n,k}=\sum_{i=0}^n\binom{n}{i}A_{i,k}\gamma^{n-i}.
\end{eqnarray}

\begin{prop}\label{prop+x+binomial+array}
Assume that $[A^\circ_{n,k}]_{n,k}$ is the $\gamma$-binomial
transformation of $[A_{n,k}]_{n,k}$.
\begin{itemize}
  \item [\rm (i)]
  If $[A_{n,k}]_{n,k}$ is $\textbf{x}$-TP$_r$ in $\mathbb{R}[\textbf{x}]$, then
  $[A^\circ_{n,k}]_{n,k}$ is $(\textbf{x},\gamma)$-TP$_r$.
   \item [\rm (ii)]
  We have
\begin{eqnarray}
A^\circ_n(q):=\sum_{k\geq0}A^\circ_{n,k}q^k=\sum_{i=0}^n\binom{n}{i}A_i(q)\gamma^{n-i}.
\end{eqnarray}
  \item [\rm (iii)]
If the array $[A_{n,k}]_{n,k}$ in $\mathbb{R}[\textbf{x}]$ satisfies
the following recurrence relation:
 \begin{eqnarray}
 A_{n,k}=r_{k-1}(\textbf{x})A_{n-1,k-1}+s_k(\textbf{x})A_{n-1,k}+t_{k+1}(\textbf{x})A_{n-1,k+1}
 \end{eqnarray}
with $A_{n,k}=0$ unless $0\le k\le n$ and $A_{0,0}=1$, then
$[A^\circ_{n,k}]_{n,k}$ satisfies the following recurrence relation:
\begin{eqnarray}  A^\circ_{n,k}=r_{k-1}(\textbf{x})A^\circ_{n-1,k-1}+(\gamma+s_k(\textbf{x}))A^\circ_{n-1,k}+t_{k+1}(\textbf{x})A^\circ_{n-1,k+1}
 \end{eqnarray}
with $A^\circ_{n,k}=0$ unless $0\le k\le n$ and $A^\circ_{0,0}=1$.
  \item [\rm (iv)]
If $(A_{n,0})_{n\geq0}$ is an $\textbf{x}$-Stieltjes moment
sequence, then $(A^\circ_{n,0})_{n\geq0}$ is an
$(\textbf{x},\gamma)$-Stieltjes moment sequence.
\end{itemize}
\end{prop}

\begin{proof}
(i) Let
$$\binom{n}{k}\gamma^{n-k}=B_{n,k}(\gamma)$$ for $n\geq k\geq0$. Clearly,
$[B_{n,k}(\gamma)]_{n,k\geq0}$ is an array satisfying the recurrence
relation
\begin{equation}\label{recurrence+Binom+matrx}
B_{n,k}(\gamma)=\gamma B_{n-1,k}(\gamma)+B_{n-1,k-1}(\gamma),
\end{equation}
where $B_{n,k}(\gamma)=0$ unless $0\le k\le n$ and
$B_{0,0}(\gamma)=1$. Thus the $\gamma$-binomial transformation of
$[A_{n,k}]_{n,k}$ is equivalent to the decomposition
\begin{equation}\label{Id+x+binom+array}
[A^\circ_{n,k}]_{n,k}=[B_{n,k}(\gamma)]_{n,k}[A_{n,k}]_{n,k}.
\end{equation}
Obviously, total positivity of the Pascal triangle implies that
$[B_{n,k}(\gamma)]_{n,k\geq0}$ is $\gamma$-TP. Then applying the
classical Cauchy-Binet formula to (\ref{Id+x+binom+array}), we
immediately get (i).

(ii) Clearly, for $n\geq0$, we have
\begin{eqnarray*}
A^\circ_n(q)=\sum_{i=0}^nA^\circ_{n,i}q^i
=\sum_{i=0}^n\sum_{k=0}^n\binom{n}{k}A_{k,i}\gamma^{n-k}q^i
=\sum_{k=0}^n \binom{n}{k} A_k(q)\gamma^{n-k}.
\end{eqnarray*}

(iii) We have \begin{eqnarray*}
A^\circ_{n,k}&=&\sum_{i=0}^n\binom{n}{i}A_{i,k}\gamma^{n-i}=\sum_{i=0}^nB_{n,i}(\gamma)A_{i,k}=\sum_{i=0}^n[\gamma B_{n-1,i}(\gamma)+B_{n-1,i-1}(\gamma)]A_{i,k}\\
&=&\gamma A^\circ_{n-1,k}+\sum_{i=0}^{n-1}B_{n-1,i}(\gamma)A_{i+1,k}\\
&=&r_{k-1}(\textbf{x})A^\circ_{n-1,k-1}+(\gamma+s_k(\textbf{x}))A^\circ_{n-1,k}+t_{k+1}(\textbf{x})A^\circ_{n-1,k+1}.
\end{eqnarray*}
It is obvious that $A^\circ_{n,k}=0$ unless $0\le k\le n$ and
$A^\circ_{0,0}=1$.

(iv) It follows from the decomposition (\ref{Id+x+binom+array}) and
Claim \ref{Cl} in Theorem \ref{thm+Hakel+Jacbi+main} (ii) that
\begin{eqnarray*}
[A^\circ_{i+j,0}]_{i,j}&=&[A^\circ_{n,k}]_{n,k}V
[A^\circ_{n,k}]_{n,k}^T\\
&=&[B_{n,k}(\gamma)]_{n,k}[A_{n,k}]_{n,k}V
[A_{n,k}]_{n,k}^T[B_{n,k}(\gamma)]_{n,k}^T\\
&=&[B_{n,k}(\gamma)]_{n,k}[A_{i+j,0}]_{i,j}[B_{n,k}(\gamma)]_{n,k}^T.
\end{eqnarray*}
Note that $[B_{n,k}(\gamma)]_{n,k}$ is $\gamma$-TP. Then, using the
classical Cauchy-Binet formula, we immediately deduce that
$[A^\circ_{i+j,0}]_{i,j}$ is $(\textbf{x},\gamma)$-TP when
$[A_{i+j,0}]_{i,j}$ is $\textbf{x}$-TP. Thus, if
$(A_{n,0})_{n\geq0}$ is $\textbf{x}$-SM, then is
$(A^\circ_{n,0})_{n\geq0}$ is $(\textbf{x},\gamma)$-SM. This
completes the proof.
\end{proof}

\begin{rem}
The (iv) of Proposition \ref{prop+x+binomial+array} can also be
proved by a different method in \cite{Zhu19}. But it seems proof
here is more natural.
\end{rem}

\section{Tridiagonal matrices and $\textbf{x}$-Stieltjes moment sequences}

The total positivity of the tridiagonal matrix
\begin{equation*}\label{J-eq}
J(\textbf{r},\textbf{s},\textbf{t})=\left[
\begin{array}{ccccc}
s_0 & r_0 &  &  &\\
t_1 & s_1 & r_1 &\\
 & t_2 & s_2 &r_2 &\\
& & \ddots&\ddots & \ddots \\
\end{array}\right]
\end{equation*} plays an
important role in Theorem \ref{thm+Hakel+Jacbi+main}. Thus, we will
present some criteria for its total positivity.

By the Laplace expansion, if each element of
$J(\textbf{r},\textbf{s},\textbf{t})$ is nonnegative, then it is not
hard to get that the total positivity of
$J(\textbf{r},\textbf{s},\textbf{t})$ is equivalent to that all
submatrices with the contiguous rows and the same columns are
totally positive.  The following positivity result about
perturbation for tridiagonal matrices is interesting and very
important.
\begin{prop}\label{prop+addition+tri}
Assume for $n\in \mathbb{P}$ that $\{a_n, b_n, c_n, a'_n, b'_n,
c'_n,b_n-b'_n,c_n-c'_n\}\subseteq \mathbb{R}^{\geq}[\textbf{x}]$.
If the matrix
\begin{eqnarray*}\left[
      \begin{array}{cccc}
        a_1 & b_1 &  &  \\
        c_1 & a_2 & b_2 &  \\
         & c_2& a_3 & \ddots \\
         &  & \ddots & \ddots \\
      \end{array}
    \right]
\end{eqnarray*}
is $\textbf{x}$-TP$_r$ in $\mathbb{R}[\textbf{x}]$, then so is the
tridiagonal matrix
\begin{eqnarray*} \left[
      \begin{array}{cccc}
        a_1+a'_1 & b_1-b'_1 &  &  \\
        c_1-c'_1 & a_2+a'_2 & b_2-b'_2 &  \\
         & c_2-c'_2& a_3+a'_3 & \ddots \\
         &  & \ddots & \ddots \\
      \end{array}
    \right].
\end{eqnarray*}
\end{prop}
\begin{proof}
Assume that \begin{eqnarray*} T_n=\left[
      \begin{array}{ccccc}
        a_1+a'_1 & b_1-b'_1 &  &  \\
        c_1-c'_1 & a_2+a'_2 & b_2-b'_2 &  \\
         & \ddots & \ddots & \ddots \\
           &&c_{n-2}-c'_{n-2}&a_{n-1}+a'_{n-1}&b_{n-1}-b'_{n-1}\\
        & &&c_{n-1}-c'_{n-1}&a_n+a'_n
      \end{array}
    \right].
\end{eqnarray*}
In order to prove that the tridiagonal matrix
\begin{eqnarray*} \left[
      \begin{array}{cccc}
        a_1+a'_1 & b_1-b'_1 &  &  \\
        c_1-c'_1 & a_2+a'_2 & b_2-b'_2 &  \\
         & c_2-c'_2& a_3+a'_3 & \ddots \\
         &  & \ddots & \ddots \\
      \end{array}
    \right]
\end{eqnarray*} is $\textbf{x}$-TP$_r$, it suffices to prove that $T_n$ is $\textbf{x}$-TP$_r$ for all
$n\geq1$. In what follows, we proceed by induction on the number
$$|\{i:a'_i\not\equiv0\}\bigcup \{i:b'_i\not\equiv0\}\bigcup
\{i:c'_i\not\equiv0\}|.$$ Since the matrix
\begin{eqnarray*}\left[
      \begin{array}{cccc}
        a_1 & b_1 &  &  \\
        c_1 & a_2 & b_2 &  \\
         & c_2& a_3 & \ddots \\
         &  & \ddots & \ddots \\
      \end{array}
    \right]
\end{eqnarray*}
is $\textbf{x}$-TP$_r$, its minors are nonnegative. Thus it is true
for
$$|\{i:a'_i\not\equiv0\}\bigcup \{i:b'_i\not\equiv0\}\bigcup
\{i:c'_i\not\equiv0\}|=0.$$ Suppose that it holds for
$|\{i:a'_i\not\equiv0\}\bigcup \{i:b'_i\not\equiv0\}\bigcup
\{i:c'_i\not\equiv0\}|=k-1$. Let us consider the next step for
$|\{i:a'_i\not\equiv0\}\bigcup \{i:b'_i\not\equiv0\}\bigcup
\{i:c'_i\not\equiv0\}|=k$. Without loss of generality, denote the
new positive element by $a'_m$, $b'_m$ or $c'_m$. If one minor of
order $\leq r$ of $T_n$ does not contain the new element, then it
belongs to $\mathbb{R}^{\geq}[\textbf{x}]$ by the inductive
hypothesis. Thus it suffices to prove the minors of order $\leq r$
containing the new element belong to
$\mathbb{R}^{\geq}[\textbf{x}]$.  Then it suffices to consider
submatrices with contiguous rows and the same column in the
following three cases.

Case $1$. Assume that the new positive element is $a'_m$. By
dividing the row containing $a_m+a'_m$ into two rows, we have
\begin{eqnarray*} \det\left[
      \begin{array}{cccc}
       \ddots & \ddots &  &  \\
        \ddots & a_m+a'_m & \ddots &  \\
         & \ddots& \ddots & \ddots \\
      \end{array}
    \right]
    &=&\det\left[
      \begin{array}{cccc}
       \ddots & \ddots &  &  \\
        \ddots & a_m & \ddots &  \\
         & \ddots& \ddots & \ddots \\
      \end{array}
    \right]
    +
    \det\left[
      \begin{array}{cccc}
       \ddots & \ddots &  &  \\
        0 & a'_m & 0 &  \\
         & \ddots& \ddots & \ddots \\
      \end{array}
    \right]\\
    &=&M_1+a'_mM_2\\
    &\geq_{\textbf{x}}&0
\end{eqnarray*}
where both $M_1$ and $M_2$ are minors in the inductive hypothesis.

Case $2$. Assume that the new positive element is $b'_m$. By
dividing the row containing $b_m-b'_m$ into two rows, we have
\begin{eqnarray*} &&\det\left[
      \begin{array}{cccccc}
       \ddots&\ddots&  &  &  &  \\
      \ddots& \ddots & \ddots &  &  \\
      & b_m-b'_m & \ddots & \ddots \\
         & & \ddots & \ddots \\
      \end{array}
    \right]\\
    &=&\det\left[
      \begin{array}{cccccc}
       \ddots&\ddots&  &  &  &  \\
      \ddots& \ddots & \ddots &  &  \\
      & b_m & \ddots & \ddots \\
         & & \ddots & \ddots \\
      \end{array}
    \right]
    +
    \det\left[
      \begin{array}{cccccc}
       \ddots&\ddots&  &  &  &  \\
      \ddots& \ddots & \ddots &  &  \\
      & -b'_m & 0 &  \\
         & & \ddots & \ddots \\
      \end{array}
    \right]\\
    &=&M_1+b'_mM_2\\
    &\geq_{\textbf{x}}&0
\end{eqnarray*}
where both $M_1$ and $M_2$ are minors in the inductive hypothesis.

Case $3$. Assume that the new positive element is $c'_m$, which is
similar to the Case $2$.

The proof is complete.
\end{proof}

\begin{rem} The result in Proposition \ref{prop+addition+tri} is very
useful for proving the total positivity of tridiagonal matrices
because it transforms a totally positive matrix to more.
\end{rem}

The next result for $\textbf{x}$ being a unique indeterminate was
proved by Chen, Liang and Wang \cite{CLW15} using diagonally
dominant matrices. Now, we give a new unified proof by Proposition
\ref{prop+addition+tri}.
\begin{prop}\label{prop+tri+TP}
Assume for $n\in \mathbb{N}$ that $\{r_n,
s_n,t_n\}\subseteq\mathbb{R}^{\geq}[\textbf{x}]$. Then the matrix
$J(\textbf{r},\textbf{s},\textbf{t})$ is $\textbf{x}$-TP under any
of the following conditions:
\begin{itemize}
\item [\rm (i)]
$s_0\geq_{\textbf{x}} r_0$ and $s_n \geq_{\textbf{x}} r_n + t_n$ for
$n\geq 1$;
\item [\rm (ii)]
$s_0\geq_{\textbf{x}} t_1$ and $s_n \geq_{\textbf{x}} r_{n-1} +
t_{n+1}$ for $n\geq 1$;
\item [\rm (iii)]
$s_0\geq_{\textbf{x}} 1$ and   $s_n \geq_{\textbf{x}} r_{n-1}
t_{n}+1$ for $n\geq 1$;
\item [\rm (iv)]
$s_0\geq_{\textbf{x}} r_0t_1$ and $s_n \geq_{\textbf{x}} r_{n}
t_{n+1}+1$ for $n\geq 1$.
\end{itemize}
\end{prop}
\begin{proof}
(i) By Claim \ref{cl-equal+tridilog} in the proof of Theorem
\ref{thm+Hakel+Jacbi+main} (ii), it suffices to prove that the
matrix
\[J'_n=\left[\begin{array}{ccccc}
s_0&1&&&\\
r_0t_1&s_1&1&&\\
&r_1t_2&s_2&1&\\
&&\ddots&\ddots&\ddots\\
&&&r_{n-1}t_n&s_n\\
\end{array}\right]
\]
is $\textbf{x}$-TP in $\mathbb{R}[\textbf{x}]$.

It follows from the decomposition
\begin{eqnarray*}
\left[
      \begin{array}{cccc}
        r_0 & 1 &  &  \\
        t_1r_0 & t_1+r_1 & 1 &  \\
         & t_2r_1 & t_2+r_2 & \ddots \\
         &  & \ddots & \ddots \\
      \end{array}
    \right]
=\left[\begin{array}{cccc}
0 & 1 &  & \\
 & t_1 & 1 & \\
 &  & t_2 & \ddots\\
 &  &  & \ddots\\
\end{array}\right]
\left[\begin{array}{cccc}
1 &  &  &\\
r_0 & 1 & &\\
 & r_1 & 1 &\\
 & & \ddots & \ddots\\
\end{array}\right]
\end{eqnarray*}
that
\begin{eqnarray*}\left[
      \begin{array}{cccc}
        r_0 & 1 &  &  \\
        t_1r_0 & t_1+r_1 & 1 &  \\
         & t_2r_1 & t_2+r_2 & \ddots \\
         &  & \ddots & \ddots \\
      \end{array}
    \right]
\end{eqnarray*}
is $\textbf{x}$-TP since $(r_n)_{n\geq0}$ and $(t_n)_{n\geq1}$
belong to $\mathbb{R}^{\geq}[\textbf{x}]$. Thus
\[\left[\begin{array}{ccccc}
s_0&1&&&\\
r_0t_1&s_1&1&&\\
&r_1t_2&s_2&1&\\
&&\ddots&\ddots&\ddots\\
&&&r_{n-1}t_n&s_n\\
\end{array}\right]
\] is $\textbf{x}$-TP by taking $a_1'=s_0-r_0,a_{n+1}'=s_n-r_n-t_n, b_m'=c_n'=0$ for $n\geq1$ in Proposition \ref{prop+addition+tri}.

Similarly, (ii), (iii) and (iv) can respectively be proved by the
next decompositions
\begin{eqnarray*}
\left[
      \begin{array}{cccc}
        t_1 & 1 &  &  \\
        t_1r_0 & t_2+r_0 & 1 &  \\
         & t_2r_1 & t_3+r_1 & \ddots \\
         &  & \ddots & \ddots \\
      \end{array}
    \right]
&=& \left[\begin{array}{cccc}
1 &  &  &\\
r_0 & 1 & &\\
 & r_1 & 1 &\\
 & & \ddots & \ddots\\
\end{array}\right]
\left[\begin{array}{cccc}
t_1 & 1 &  & \\
 & t_2 & 1 & \\
 &  & t_3 & \ddots\\
 &  &  & \ddots\\
\end{array}\right],
\end{eqnarray*}
\begin{eqnarray*}
\left[
      \begin{array}{cccc}
        1 & 1 &  &  \\
        t_1r_0 & t_1r_0+1 & 1 &  \\
         & t_2r_1 & t_2r_1+ & \ddots \\
         &  & \ddots & \ddots \\
      \end{array}
    \right]
&=&\left[\begin{array}{cccc}
0 & 1 &  & \\
 & t_1r_0 & 1 & \\
 &  & t_2r_1 & \ddots\\
 &  &  & \ddots\\
\end{array}\right]
\left[\begin{array}{cccc}
1 &  &  &\\
1 & 1 & &\\
 & 1 & 1 &\\
 & & \ddots & \ddots\\
\end{array}\right],\\
\left[
      \begin{array}{cccc}
        t_1r_0 & 1 &  &  \\
        t_1r_0 & 1+t_2r_1 & 1 &  \\
         & t_2r_1 & 1+t_3r_2 & \ddots \\
         &  & \ddots & \ddots \\
      \end{array}
    \right]
&=&\left[\begin{array}{cccc}
0 & 1 &  & \\
 & 1 & 1 & \\
 &  & 1 & \ddots\\
 &  &  & \ddots\\
\end{array}\right]
\left[\begin{array}{cccc}
1 &  &  &\\
t_1r_0 & 1 & &\\
 & t_2r_1 & 1 &\\
 & & \ddots & \ddots\\
\end{array}\right].
\end{eqnarray*}
\end{proof}

Based on Theorem \ref{thm+Hakel+Jacbi+main} and Proposition
\ref{prop+tri+TP} (i),  we obtain the following result which uses
Jacobi continued fractions to prove x-Stieltjes moment property

\begin{thm}\label{thm+S+C}
Let $\{\lambda_n(\textbf{x}),
\mu_{n}(\textbf{x}),T_n(\textbf{x})\}\subseteq
\mathbb{R}^{\geq}[\textbf{x}]$ for $n\in \mathbb{N}$ and
$$\sum_{n\geq0}
T_n(\textbf{x})z^n=\J[\lambda_n(\textbf{x}),\mu_{n+1}(\textbf{x});z]_{n\geq0}.$$
If there exists polynomials $\alpha_n(\textbf{x})$ and
$\beta_n(\textbf{x})$ in $\mathbb{R}^{\geq}[\textbf{x}]$ such that
$$\lambda_n(\textbf{x})=\beta_n(\textbf{x})+\alpha_{2n}(\textbf{x})+\alpha_{2n-1}(\textbf{x}),
\mu_{n+1}(\textbf{x})=\alpha_{2n}(\textbf{x})\alpha_{2n+1}(\textbf{x}),$$
then $T_n(\textbf{x})$ form an $\textbf{x}$-SM and
$3$-$\textbf{x}$-LCX sequence for $n\in \mathbb{N}$.
\end{thm}

\begin{proof}
In order to prove that $T_n(\textbf{x})$ form an $\textbf{x}$-SM and
$3$-$\textbf{x}$-LCX sequence, by Theorem
\ref{thm+Hakel+Jacbi+main}, it suffices to prove the corresponding
tridiagonal matrix
\[\left[\begin{array}{ccccc}
\lambda_0(\textbf{x})&\alpha_{0}(\textbf{x})&&&\\
\alpha_{1}(\textbf{x})&\lambda_1(\textbf{x})&\alpha_{2}(\textbf{x})&&\\
&\alpha_{3}(\textbf{x})&\lambda_2(\textbf{x})&\alpha_{4}(\textbf{x})&\\
&&\ddots&\ddots&\ddots\\
&&&\alpha_{2n-1}(\textbf{x})&\lambda_n(\textbf{x})\\
\end{array}\right]
\]
is $\textbf{x}$-TP in $\mathbb{R}[\textbf{x}]$. Clearly,
\begin{eqnarray*}
\lambda_0(\textbf{x})=\beta_0(\textbf{x})+\alpha_{0}(\textbf{x})&\geq_{x}&\alpha_{0}(\textbf{x}),\\
\lambda_n(\textbf{x})=\beta_n(\textbf{x})+\alpha_{2n}(\textbf{x})+\alpha_{2n-1}(\textbf{x})&\geq_{x}&\alpha_{2n}(\textbf{x})+\alpha_{2n-1}(\textbf{x})\quad\text{for}\quad
n\geq1.
\end{eqnarray*}
It follows from Proposition \ref{prop+tri+TP} (i) that we get the
desired total positivity.
\end{proof}
The following result will play an important role in the proof of
Theorem \ref{thm+Ring+PSE+SM}.

\begin{thm}\label{thm+Gass}
Assume
\begin{eqnarray*}
\sum_{n\geq0}G_n(a,b,c)z^n=1+\sum_{n\geq1}z^n\prod_{k=0}^{n-1}\frac{a+bk}{1-c(k+1)z}.
\end{eqnarray*}
Then we have continued fraction expansions
\begin{eqnarray*}
\sum_{n\geq0}G_n(a,b,c)z^n&=&\s[\alpha_{2n},\alpha_{2n+1};z]_{n\geq0}\\
&=&\J[\alpha_{2n}+\alpha_{2n-1},\alpha_{2n}\alpha_{2n+1};z]_{n\geq0},
\end{eqnarray*}
where $\alpha_{2n}(a)=a+nb$ and $\alpha_{2n+1}(a)=(c+b)(1+n)$ for
$n\geq0$. In addition, $G_n(a,b,c)$ form an $\textbf{x}$-SM and
$3$-$\textbf{x}$-LCX sequence with $\textbf{x}=(a,b,c)$ for $n\in
\mathbb{N}$.
\end{thm}
\begin{proof}
By Theorem \ref{thm+S+C}, it suffices to prove
\begin{eqnarray*}
\sum_{n\geq0}G_n(a,b,c)z^n
&=&\J[\alpha_{2n}+\alpha_{2n-1},\alpha_{2n}\alpha_{2n+1};z]_{n\geq0}
\end{eqnarray*}
with $\alpha_{2n}(a)=a+nb$ and $\alpha_{2n+1}(a)=(c+b)(1+n)$ for
$n\geq0$.

Define the function
\begin{eqnarray*}
F(a,z):=1+\sum_{n\geq1}z^n\prod_{k=0}^{n-1}\frac{a+bk}{1-c(k+1)z},
\end{eqnarray*}
which implies
 \begin{eqnarray}\label{eq+F}
 F(a,z)=1+\frac{az}{1-cz}\,F(a+b,\frac{z}{1-cz}).
\end{eqnarray}
Assume that
\begin{eqnarray*}
F(a,z)&=&\frac{1}{1-\frac{\alpha_0(a)z}{1-\frac{\alpha_1(a)z}{1-\frac{\alpha_2(a)z}{1-\ldots}}}}.
\end{eqnarray*}
It follows from the contraction formula (\ref{contraction}) that
\begin{eqnarray}
F(a,z) &=&\DF{1}{1- \alpha_0(a)z-\DF{\alpha_0(a)\alpha_1(a)z^2}{1-
(\alpha_1(a)+\alpha_2(a))z-\DF{\alpha_2(a)\alpha_3(a)z^2}{1-
(\alpha_3(a)+\alpha_4(a))z-\ldots}}}\label{Eq+S2}.
\end{eqnarray}
Combining (\ref{eq+F}) and (\ref{Eq+S2}), we have
\begin{eqnarray*}
&&F(a,z)\\
&=&1+\frac{az}{1-cz}\,F(a+b,\frac{z}{1-cz})\\
&=&1+\DF{az}{1-
(c+\alpha_0(a+b))z-\DF{\alpha_0(a+b)\alpha_1(a+b)z^2}{1-
(c+\alpha_1(a+b)+\alpha_2(a+b))z-\DF{\alpha_2(a+b)\alpha_3(a+b)z^2}{1-\ldots}}}.
\end{eqnarray*}
On the other hand, by the contraction formula
(\ref{contraction+decom}), we also have
\begin{eqnarray*}
F(a,z) &=&1+\DF{\alpha_0(a)z}{1-
(\alpha_0(a)+\alpha_1(a))z-\DF{\alpha_1(a)\alpha_2(a)z^2}{1-
(\alpha_2(a)+\alpha_3(a))z-\DF{\alpha_3(a)\alpha_4(a)z^2}{1-
(\alpha_4(a)+\alpha_5(a))z-\ldots}}}.
\end{eqnarray*}
Thus, we get equations
\begin{eqnarray*} \alpha_0(a)&=&a\\
\alpha_0(a)+\alpha_1(a)&=&c+\alpha_0(a+b)\\
\alpha_1(a)\alpha_2(a)&=&\alpha_0(a+b)\alpha_1(a+b)\\
\alpha_2(a)+\alpha_3(a)&=&c+\alpha_1(a+b)+\alpha_2(a+b)\\
&\cdots&.
\end{eqnarray*}
Solving equations, we get $\alpha_{2n}(a)=a+nb$ and
$\alpha_{2n+1}(a)=(c+b)(1+n)$ for $n\geq0$.
\end{proof}

\section{Proof of Theorem \ref{thm+Ring+PSE+SM}}

In this section, in order to apply previous results to prove Theorem
\ref{thm+Ring+PSE+SM}, the key is to obtain the continued fraction
expansion for the generating function $\sum_{n\geq0}T_n(q)t^n$.

We first present an important relationship between two different
arrays.
\begin{lem}\label{lem+shit}
Let $\{a_0,a_2,b_0,b_1,b_2,\lambda\}\subseteq \mathbb{R}$. Assume
that an array $[A_{n,k}]_{n,k\geq0}$ satisfies the recurrence
relation:
\begin{equation}\label{recurrence relation+A}
A_{n,k}=(a_0n-\lambda b_1k+a_2)A_{n-1,k}+(b_0n+b_1k+b_2)A_{n-1,k-1},
\end{equation}
where $A_{n,k}=0$ unless $0\le k\le n$ and $A_{0,0}=1$. If a
polynomial $B_n(q)=A_n(q+\lambda)$, then the coefficient array
$[B_{n,k}]_{n,k\geq0}$ of $B_n(q)$ satisfies the recurrence
relation:
\begin{equation*}
B_{n,k}=[(a_0+\lambda b_0)n+\lambda
b_1k+a_2+\lambda(b_1+b_2)]B_{n-1,k}+(b_0n+b_1k+b_2)B_{n-1,k-1},
\end{equation*}
where  $B_{0,0}=1$ and $B_{n,k}=0$ unless $0\le k\le n$.
\end{lem}
\begin{proof}
By the recurrence relation (\ref{recurrence relation+A}), we have
\begin{eqnarray*}
A_n(q) &=& [a_0n+a_2+(b_0n+b_1+b_2)q]A_{n-1}(q)+b_1(-\lambda
+q)qA'_{n-1}(q).
\end{eqnarray*}
Then using $B_n(q)=A_n(q+\lambda)$, we get
\begin{eqnarray*}
B_n(q) &=& [(a_0+\lambda
b_0)n+a_2+\lambda(b_1+b_2)+(b_0n+b_1+b_2)q]B_{n-1}(q)+b_1(\lambda
+q)qB'_{n-1}(q)
\end{eqnarray*}
for $n\geq1$. This implies that $[B_{n,k}]_{n,k\geq0}$ satisfies the
recurrence relation:
\begin{equation*}
B_{n,k}=[(a_0+\lambda b_0)n+\lambda
b_1k+a_2+\lambda(b_1+b_2)]B_{n-1,k}+(b_0n+b_1k+b_2)B_{n-1,k-1},
\end{equation*}
where $B_{n,k}=0$ unless $0\le k\le n$ and $B_{0,0}=1$.
\end{proof}

\begin{rem}
For a triangle $[T_{n,k}]_{n,k}$, define its {\it reciprocal
triangle} $[T^*_{n,k}]_{n,k}$ by
$$T^*_{n,k}=T_{n,n-k},\quad 0\leq k \leq n.$$ In addition, we have $T^*_n(q)=q^nT_n(\frac{1}{q})$.
\end{rem}

We present \textbf{the proof of Theorem \ref{thm+Ring+PSE+SM} as
follows:}

\begin{proof}
(i) From the recurrence relation in (i), we get
\begin{eqnarray}\label{eq+fst}
T_n(q)&=&(a_0n+a_2-a_0)T_{n-1}(q)+(b_0n+b_2-b_0)qT_{n-1}(q)\nonumber\\
&=&[(a_0n+a_2-a_0)+(b_0n+b_2-b_0)q]T_{n-1}(q)\nonumber\\
&=&\prod_{k=1}^n[(a_0+b_0q)k+a_2-a_0+(b_2-b_0)q]
\end{eqnarray}
for $n\geq1$. In Theorem \ref{thm+Gass}, if we take $$
a=a_2+b_2q,\quad \  b=a_0+b_0q\quad\text{and}\, c=0,$$ then we
immediately have
  \begin{eqnarray}\label{eq+fst+CF}
  \sum\limits_{n=0}^{\infty}T_{n}(q) z^n
  &=&\s[\alpha_{2n},\alpha_{2n+1};z]_{n\geq0}\nonumber\\
  &=&\J[\alpha_{2n-1}+\alpha_{2n},\alpha_{2n}\alpha_{2n+1};z]_{n\geq0}
\end{eqnarray}
and $(T_n(q))_{n\geq0}$ is an $\textbf{x}$-SM sequence with
$\textbf{x}=(a_0,a_2,b_0,b_2,q)$, where
$\alpha_{2n}=a_2+b_2q+n(a_0+b_0q)$ and
$\alpha_{2n+1}=(a_0+b_0q)(1+n)$.

(ii) First let us consider a degenerated array
$[A_{n,k}]_{n,k\geq0}$ satisfying the recurrence relation
$$A_{n,k}=a_0b_1(n-k-1)A_{n-1,k}+[b_0(n-1)+b_1(k-1)+b_2]A_{n-1,k-1}$$ for $n\geq1$ with $A_{0,0}=1$.
It is easy to find $A_{n,k}=0$ for $n\neq k$. Then by induction on
$n$, we immediately get for $n\geq1$ that
\begin{eqnarray*}
A_{n,k}&=&0\quad \text{for} \quad n\neq k,\\
A_{n,n}&=&\prod_{k=0}^{n-1}[b_2+(b_0+b_1)k].
\end{eqnarray*}
Clearly, this triangle $[A_{n,k}]_{n,k}$ is degenerated to a
diagonal sequence. So for $n\geq1$ the generating function
$$A_n(q)=q^nA_{n,n}=q^n\prod_{k=0}^{n-1}[b_2+(b_0+b_1)k].$$
For the array $[T_{n,k}]_{n,k\geq0}$ in (ii), by taking
$\lambda=-a_0$ in Lemma \ref{lem+shit}, we get
$$T_n(q)=A_n(q+a_0)$$ for $n\geq1$. So
$$T_n(q)=(q+a_0)^n\prod_{k=0}^{n-1}[b_2+(b_0+b_1)k]$$
for $n\geq1$. It follows from Theorem \ref{thm+Gass} that
\begin{eqnarray*}
\sum_{n\geq0}T_n(q)z^n&=&\s[\alpha_{2n},\alpha_{2n+1};z]_{n\geq0}\\
&=&\J[\alpha_{2n-1}+\alpha_{2n},\alpha_{2n}\alpha_{2n+1};z]_{n\geq0}
\end{eqnarray*}
with $\alpha_{2n}=(b_2+n(b_0+b_1))(q+a_0)$ and
$\alpha_{2n+1}=(b_0+b_1)(q+a_0)(1+n)$, and $(T_n(q))_{n\geq0}$ is an
$\textbf{x}$-SM sequence with $\textbf{x}=(a_0,b_0,b_1,b_2,q)$.

(iii) By Proposition \ref{prop+x+binomial+array} and \cite[Theorem
1]{Zhu19}, $[T_{n,k}]_{n,k\geq0}$ can be considered as the
$a_2$-binomial transformation of $[\widetilde{T}_{n,k}]_{n,k\geq0}$,
where $[\widetilde{T}_{n,k}]_{n,k\geq0}$ satisfies the recurrence
relation
\begin{equation*}
\widetilde{T}_{n,k}=a_1k\widetilde{T}_{n-1,k}+(b_1k+b_2-b_1)\widetilde{T}_{n-1,k-1}
\end{equation*}
for $n\geq1$ and $\widetilde{T}_{0,0}=1$. Let
$f_k=\sum_{n\geq0}\widetilde{T}_{n,k}z^n$ for $k\geq0$ and
$F(z)=\sum_{k\geq0}f_kq^k$. By the recurrence relation, we get
$$f_k=a_1kz f_k+(b_1k+b_2-b_1)zf_{k-1},$$
which implies
$$f_k=\prod_{i=1}^k\frac{(b_1i+b_2-b_1) z}{1-a_1i z}$$
for $k\geq1$. Thus, we have
\begin{eqnarray*}
F(z)&=&1+\frac{b_2qz}{1-a_1z}+\frac{b_2(b_1+b_2)q^2z^2}{(1-a_1z)(1-2a_1z)}+\cdots\\
&=&1+\sum_{k\geq1}z^k\prod_{i=1}^k\frac{(b_1i+b_2-b_1)q}{1-a_1iz}.
\end{eqnarray*}
It follows from Theorem \ref{thm+Gass} that
\begin{eqnarray*}
  \sum\limits_{n=0}^{\infty}\widetilde{T}_{n}(q) z^n&=&\J[\alpha_{2n-1}+\alpha_{2n},\alpha_{2n}\alpha_{2n+1};z]_{n\geq0}\\
  &=&\s[\alpha_{2n},\alpha_{2n+1};z]_{n\geq0},
\end{eqnarray*}
where $\alpha_{2n}=(nb_1+b_2)q$ and  $\alpha_{2n+1}=(n+1)(a_1+b_1q)$
for $n\geq0$. In addition, $(\widetilde{T}_n(q))_{n\geq0}$ is an
$\textbf{x}$-SM sequence. Thus by Proposition
\ref{prop+x+binomial+array}, we have
\begin{eqnarray}\label{Stir+Sec+CF}
  \sum\limits_{n=0}^{\infty}T_{n}(q) z^n&=&\J[a_2+\alpha_{2n-1}+\alpha_{2n},\alpha_{2n}\alpha_{2n+1};z]_{n\geq0},
\end{eqnarray}
and $(T_n(q))_{n\geq0}$ is an $\textbf{x}$-SM sequence with
$\textbf{x}=(a_1,a_2,b_1,b_2,q)$.

(iv) Clearly, the array $[T_{n,k}]_{n,k\geq0}$ in (iv) is the
reciprocal array of that in (iii). By (\ref{Stir+Sec+CF}), we have
\begin{eqnarray}\label{rec+ST}
  \sum\limits_{n=0}^{\infty}T_{n}(q) z^n
  &=&\J[b_2q+\alpha_{2n-1}+\alpha_{2n},\alpha_{2n}\alpha_{2n+1};z]_{n\geq0},
\end{eqnarray}
where $\alpha_{2n}=na_0+a_2$ and  $\alpha_{2n+1}=(n+1)(a_0+b_0q)$
for $n\geq0$. By Theorem \ref{thm+S+C}, $(T_n(q))_{n\geq0}$ is an
$\textbf{x}$-SM sequence with $\textbf{x}=(a_0,a_2,b_0,b_2,q)$

(v) For the array $[T_{n,k}]_{n,k}$, let
$B_n(q)=T_n(q+\frac{a_1}{b_0})$. By Lemma \ref{lem+shit}, we have
the coefficient array $[B_{n,k}]_{n,k\geq0}$ of $B_n(q)$ satisfying
the recurrence relation
$$B_{n,k}=\left[a_1n-a_1k+a_2+\frac{a_1(b_2-b_0)}{b_0}\right]B_{n-1,k}+(b_0n-b_0k+b_2)B_{n-1,k-1}$$
for $n\geq1$, which is a special case in (iv). It follows from the
continued fraction expansion in (\ref{rec+ST}) that
 \begin{eqnarray}\label{CF+EF}
  &&\sum\limits_{n=0}^{\infty}T_{n}(q) z^n=\sum\limits_{n=0}^{\infty}B_{n}(q-\frac{a_1}{b_0})
  z^n\nonumber\\
 &=&\J[n(a_1+b_0q)+a_2+b_2q,(n+1)(na_1b_0+a_2b_0+a_1b_2)q;z]_{n\geq0}\\
&=&\begin{cases}\s[na_1+a_2,(n+1)b_0q;z]_{n\geq0}
&\text{for}\quad  b_2=0\\
\s[(nb_0+b_2)q,(n+1)a_1;z]_{n\geq0}&\text{for}\quad a_2=0\\
\s[(n+1)b_0q,na_1+a_1+a_2;z]_{n\geq0}&\text{for}\quad  b_2=b_0\\
\s[(n+1)a_1,(nb_0+b_0+b_2)q;z]_{n\geq0}&\text{for} \quad
a_2=a_1.\end{cases}\nonumber
\end{eqnarray}
By Theorem \ref{thm+S+C}, $(T_n(q))_{n\geq0}$ is an $\textbf{x}$-SM
sequence with $\textbf{x}=(a_1,a_2,b_0,b_2,q)$  for
  $0\in\{a_2,b_2,a_1-a_2,b_0-b_2\}$.

(vi) For the array $[T_{n,k}]_{n,k}$ in (vi), in what follows, we
will prove
\begin{eqnarray}\label{CF++}
\sum\limits_{n=0}^{\infty}T_{n}(q)
z^n=\J[(a_1n+a_2)(q+\frac{2b_0}{a_1}),\frac{b_0(a_1n+2a_2)(n+1)(q+\frac{2b_0}{a_1})}{2};z]_{n\geq0},
\end{eqnarray} which immediately implies that $(T_n(q))_{n\geq0}$ is
an $\textbf{x}$-SM sequence with $\textbf{x}=(a_1,a_2,b_0,q)$ by
taking $s_n=(a_1 n+a_2)(q+\frac{2b_0}{a_1})$,
$r_n=(\frac{na_1}{2}+a_2)(q+\frac{2b_0}{a_1})$ and $t_n=b_0n$ for
$n\geq0$ in Theorem \ref{thm+Hakel+Jacbi+main} (ii) and Proposition
\ref{prop+tri+TP} (i).

Let $S^*_n(q+\frac{2b_0}{a_1})=T_n(q)$ and
$S^*_n(q)=\sum_{k\geq0}S^*_{n,k}q^k$ for $n\geq0$. Then for
(\ref{CF++}) it suffices to prove
\begin{eqnarray}\label{CF+++}
\sum\limits_{n=0}^{\infty}S^*_{n}(q)
z^n=\J[(a_1n+a_2)q,\frac{b_0(a_1n+2a_2)(n+1)q}{2};z]_{n\geq0}.
\end{eqnarray}
Moreover, by Lemma \ref{lem+shit}, the array
$[S^*_{n,k}]_{n,k\geq0}$ satisfies the recurrence relation:
\begin{equation}
S^*_{n,k}=b_0\left(2k-n+1\right)S^*_{n-1,k}+[a_1\,(n-k)+a_2]S^*_{n-1,k-1}
\end{equation}
for $0\leq k\leq n$ with $S^*_{0,0}=1$ (In addition, it is easy to
prove $S^*_{n,k}=0$ for $k< (n-1)/2$). In view of its reciprocal
array, for (\ref{CF+++}), we will demonstrate
\begin{eqnarray}\label{CF++++}
\sum\limits_{n=0}^{\infty}S_{n}(q)
z^n=\J[a_1n+a_2,\frac{b_0(a_1n+2a_2)(n+1)q}{2};z]_{n\geq0},
\end{eqnarray}
where $S_n(q)$ is the row-generating function of the reciprocal
array $[S_{n,k}]_{n,k\geq0}$ satisfying the recurrence relation:
\begin{equation}\label{S+recurrence relation}
S_{n,k}=(a_1k+a_2)S_{n-1,k}+b_0(n-2k+1)S_{n-1,k-1}
\end{equation}
for $0\leq k\leq n$, where $S_{0,0}=1$. It is not hard to get
(\ref{CF++++}) from the next claim and the continued fraction
expression (\ref{CF+EF}) in (v).

\begin{cl}\footnote{This result was proved in \cite{Zhu20}. For convenience of the reader, we cite
its proof to be self-contained.} For the array $[S_{n,k}]_{n,k}$ in
(\ref{S+recurrence relation}), there exists an array
$[E_{n,k}]_{n,k}$ satisfying the recurrence relation:
\begin{equation}\label{two+recurrence relation}
E_{n,k}=(a_1k+a_2)E_{n-1,k}+[a_1(n-k)+a_2]E_{n-1,k-1}
\end{equation}
with $E_{n,k}=0$ unless $0\le k\le n$ and $E_{0,0}=1$ such that
\begin{eqnarray*}
E_n(x)=(1+x)^nS_n\left(\frac{\frac{2a_1}{b_0} x}{(1+x)^{2}}\right)
\end{eqnarray*}
for $n\geq1$.
\end{cl}
\begin{proof}
We will prove this result by in induction on $n$. Let
$\frac{2a_1}{b_0}=\lambda$. It is obvious for $n=0$. For $n=1$,
$E_1(x)=a_2+a_2x$ and $S_1(x)=a_2$. Therefore we have
$$E_1(x)=(1+x)S_1\left(\frac{\lambda x}{(1+x)^{2}}\right).$$ In the following we
assume $n\geq2$. By the inductive hypothesis, we have
\begin{eqnarray}
E'_{n-1}(x)=(n-1)(1+x)^{n-2}S_{n-1}(\frac{\lambda
x}{(1+x)^{2}})+\frac{(1+x)^{n-1}\lambda
(1-x)}{(1+x)^3}S'_{n-1}(\frac{\lambda x}{(1+x)^{2}}).
\end{eqnarray}
On the other hand, by the recurrence relations (\ref{S+recurrence
relation}) and (\ref{two+recurrence relation}), we have
\begin{eqnarray*}
S_n(x)&=&[a_2+b_0(n-1)x]S_{n-1}(x)+x(a_1-2b_0x)S'_{n-1}(x),\\
E_n(x)&=&[a_2(x+1)+a_1x(n-1)]E_{n-1}(x)+a_1x(1-x)E'_{n-1}(x).
\end{eqnarray*} Then we get
\begin{eqnarray*}
E_n(x) &=&[a_2(x+1)+a_1x(n-1)](1+x)^{n-1}S_{n-1}\left(\frac{\lambda
x}{(1+x)^{2}}\right)+a_1x(1-x)\\
&&\times\left[(n-1)(1+x)^{n-2}S_{n-1}\left(\frac{\lambda
x}{(1+x)^{2}}\right)+\frac{(1+x)^{n-1}\lambda
(1-x)}{(1+x)^3}S'_{n-1}\left(\frac{\lambda x}{(1+x)^{2}}\right)\right]\\
&=&(1+x)^n\left[a_2+\frac{a_1(n-1)x}{1+x}\right]S_{n-1}\left(\frac{\lambda
x}{(1+x)^{2}}\right)+a_1x(1-x)(1+x)^n\\
&&\times\left[\frac{(n-1)}{(1+x)^{2}}S_{n-1}\left(\frac{\lambda
x}{(1+x)^{2}}\right)+\frac{\lambda
(1-x)}{(1+x)^4}S'_{n-1}\left(\frac{\lambda
x}{(1+x)^{2}}\right)\right]\\
&=&(1+x)^n\left[a_2+b_0(n-1)\times\frac{\lambda
x}{(1+x)^2}\right]S_{n-1}\left(\frac{\lambda
x}{(1+x)^{2}}\right)+(1+x)^n\times \frac{\lambda
x}{(1+x)^{2}}\times\\
&&\left[a_1-2b_0\times\frac{\lambda x}{(1+x)^{2}}\right]\times
S'_{n-1}\left(\frac{\lambda
x}{(1+x)^{2}}\right)\\
&=&(1+x)^nS_{n}\left(\frac{\lambda x}{(1+x)^{2}}\right).
\end{eqnarray*}
This proves the claim.
\end{proof}
(vii) Obviously, the array in (vii) is the reciprocal array of that
in (vi). By (\ref{CF++}), we have the Jacobi continued fraction
expansion
\begin{eqnarray*}
\sum\limits_{n=0}^{\infty}T_{n}(q)
z^n=\J[(a_1n+a_2)(1+\frac{2b_0}{a_1}q),\frac{b_0(a_1n+2a_2)(n+1)(1+\frac{2b_0q}{a_1})q}{2};z]_{n\geq0}
\end{eqnarray*}
and $(T_n(q))_{n\geq0}$ is an $\textbf{x}$-SM sequence with
$\textbf{x}=(a_1,a_2,b_0,q)$.

The proof is complete.
\end{proof}

\section{Proof of Theorem \ref{thm+three+PSE+SM}}
\begin{proof}
(i) By Theorem \ref{thm+m+tran}, there exists an array
$[A_{n,k}]_{n,k\geq0}$ satisfying the recurrence relation
 \begin{eqnarray}
A_{n,k}=(a_1k+a_2)A_{n-1,k}+[b_1(k-1)+b_2]A_{n-1,k-1}
\end{eqnarray}
with $A_{0,0}=1$ and $A_{n,k}=0$ unless $0\le k\le n$ such that
their row-generating functions satisfy
\begin{eqnarray}\label{rel+T+A}
T_n(q)=(\lambda+dq)^nA_n(\frac{q}{\lambda+dq})
\end{eqnarray} for $n\geq0$.
It follows from (\ref{Stir+Sec+CF}) that
\begin{eqnarray}
  \sum\limits_{n=0}^{\infty}T_{n}(q) z^n&=&\J[a_2(\lambda+dq)+\alpha_{2n-1}+\alpha_{2n},\alpha_{2n}\alpha_{2n+1};z]_{n\geq0},
\end{eqnarray}
where $\alpha_{2n}=(nb_1+b_2)q$ and
$\alpha_{2n+1}=(n+1)[(a_1d+b_1)q+\lambda a_1]$ for $n\geq0$. Thus by
Theorem \ref{thm+S+C}, we get that $(T_n(q))_{n\geq0}$ is an
$\textbf{x}$-SM sequence with
$\textbf{x}=(a_1,a_2,b_1,b_2,\lambda,q)$.

(ii) By Theorem \ref{thm+m+tran}, there exists an array
$[A_{n,k}]_{n,k\geq0}$ satisfying the recurrence relation
 \begin{eqnarray}
A_{n,k}=(a_0n-a_0k-a_0+a_2)A_{n-1,k}+[b_0(n-k)+b_2]A_{n-1,k-1}
\end{eqnarray}
with $A_{0,0}=1$ and $A_{n,k}=0$ unless $0\le k\le n$ such that
their row-generating functions satisfy
\begin{eqnarray}\label{rel+T+A}
T_n(q)=(\lambda+dq)^nA_n(\frac{q}{\lambda+dq})
\end{eqnarray} for $n\geq0$.
It follows from (\ref{rec+ST}) that
\begin{eqnarray}
  \sum\limits_{n=0}^{\infty}T_{n}(q) z^n
  &=&\J[b_2q+\alpha_{2n-1}+\alpha_{2n},\alpha_{2n}\alpha_{2n+1};z]_{n\geq0},
\end{eqnarray}
where $\alpha_{2n}=(na_0+a_2)(\lambda+dq)$ and
$\alpha_{2n+1}=(n+1)[(a_0d+b_0)q+\lambda a_0]$ for $n\geq0$. By
Theorem \ref{thm+S+C}, $(T_n(q))_{n\geq0}$ is an $\textbf{x}$-SM
sequence with $\textbf{x}=(a_0,a_2,b_0,b_2,d,\lambda,q)$.

(iii) By Theorem \ref{thm+m+tran}, there exists an array
$[A_{n,k}]_{n,k\geq0}$ satisfying the recurrence relation
 \begin{eqnarray}
A_{n,k}=(a_1k+a_2)A_{n-1,k}+[b_0(n-k)+b_2]A_{n-1,k-1}
\end{eqnarray}
with $A_{0,0}=1$ and $A_{n,k}=0$ unless $0\le k\le n$ such that
their row-generating functions satisfy
\begin{eqnarray}\label{rel+T+A}
T_n(q)=(\lambda+dq)^nA_n(\frac{q}{\lambda+dq})
\end{eqnarray} for $n\geq0$.
It follows from (\ref{CF+EF}) that
\begin{eqnarray*}
  &&\sum\limits_{n=0}^{\infty}T_{n}(q) z^n\nonumber\\
 &=&\J[n(a_1(\lambda+dq)+b_0q)+a_2(\lambda+dq)+b_2q,(n+1)(na_1b_0+a_2b_0+a_1b_2)q(\lambda+dq);z]_{n\geq0}\nonumber\\
&=&\begin{cases}\s[(na_1+a_2)(\lambda+dq),(n+1)b_0q;z]_{n\geq0}
&\text{for}\quad  b_2=0\\
\s[(nb_0+b_2)q,(n+1)a_1(\lambda+dq);z]_{n\geq0}&\text{for}\quad a_2=0\\
\s[(n+1)b_0q,(na_1+a_1+a_2)(\lambda+dq);z]_{n\geq0}&\text{for}\quad  b_2=b_0\\
\s[(n+1)a_1q,(nb_0+b_0+b_2)(\lambda+dq);z]_{n\geq0}&\text{for} \quad
a_2=a_1.\end{cases}
\end{eqnarray*}
Then by Theorem \ref{thm+S+C}, $(T_n(q))_{n\geq0}$ is an
$\textbf{x}$-SM sequence with
$\textbf{x}=(a_1,a_2,b_0,b_2,d,\lambda,q)$ for
  $0\in\{a_2,b_2,a_1-a_2,b_0-b_2\}$.

\end{proof}

\section{Proof of Theorem \ref{thm+sequence+SM}}

\begin{proof}

It follows from the recurrence relation
  $$T_{n,k}=(a_0n-\mu
  b_1k+a_2)T_{n-1,k}+(b_0n+b_1k+b_2)T_{n-1,k-1}$$ that
  $$T_n(q)=[a_0n+a_2+q(b_0n+b_1+b_2)]T_{n-1}(q)+q(-\mu b_1+b_1q)T'_{n-1}(q).$$
  Setting $q=\mu$, we have
 \begin{eqnarray*}
 T_n(\mu) &=&\prod_{k=1}^n[(a_0+\mu b_0)k+a_2+\mu(b_1+b_2)]
 \end{eqnarray*} for
  $n\geq1$.
 By Theorem \ref{thm+Gass}, we immediately get that
 \begin{eqnarray*}
\sum_{n\geq0}T_n(\mu)z^n
&=&\J[\alpha_{2n}+\alpha_{2n-1},\alpha_{2n}\alpha_{2n+1};z]_{n\geq0}
\end{eqnarray*}
with $\alpha_{2n}(a)=a_0+a_2+\mu(b_0+b_1+b_2)+n(a_0+\mu b_0)$ and
$\mu_{2n+1}(a)=(a_0+\mu b_0)(1+n)$ for $n\geq0$ and
$(T_n(\mu))_{n\geq0}$ is an $\textbf{x}$-SM sequence with
$\textbf{x}=(a_0,a_2,b_0,b_1,b_2,\mu)$.

\end{proof}

\section{Examples}
In this section, we present some examples related to continued
fractions.
\begin{ex}[\textbf{Factorial numbers}]
By (\ref{eq+fst+CF}), if $a_0=b_2=q=1$ and $b_0=a_2=0$, then we
immediately have
\begin{eqnarray*}
\sum_{n\geq0} n!z^n&=&\s[n+1,n+1;z]_{n\geq0}
\end{eqnarray*}
and similarly if $a_0=a_2=b_0=q=1$ and $b_2=0$, then we have
\begin{eqnarray*}
\sum_{n\geq0}(2n-1)!!z^n&=&\s[1+2n,2(1+n);z]_{n\geq0},
\end{eqnarray*}
which were proved by Euler \cite{Eul76}. Thus the continued fraction
expansion in (\ref{eq+fst+CF}) can be looked at a generalization of
above two results and the corresponding row-generating function
$T_n(a_0,a_2,b_0,b_2,q)$ in (\ref{eq+fst}) can be viewed as a
five-variable refinement of $n!$ and $(2n-1)!!$.
\end{ex}

\begin{ex}[\textbf{Whitney numbers of the first kind}]
In \cite{CJ12}, the $r$-Whitney numbers of the first kind, denoted
by $w_{m,r}(n, k)$, satisfy the recurrence relation
\begin{equation*}
w_{m,r}(n, k)=[(n-1)m+r]w_{m,r}(n-1, k)+w_{m,r}(n-1, k-1)
\end{equation*}
with $w_{m,r}(0, 0)=1$. It reduces to the signless Stirling number
of the first kind for $m=r=1$. Let the row-generating functions
$w_{n}(m,r,q)=\sum_{k\geq0}w_{m,r}(n, k)q^k$ for $n\geq0$. It
follows from Theorem \ref{thm+Ring+PSE+SM} (i) and (\ref{eq+fst+CF})
that we have
\begin{eqnarray*}
\sum_{n\geq0}w_{n}(m,r,q)z^n&=&\J[r+q+2mn,(r+q+mn)m(n+1);z]_{n\geq0}\\
&=&\s[r+q+mn,m(n+1);z]_{n\geq0}
\end{eqnarray*}
and $(w_{n}(m,r,q))_{n\geq0}$ is an $\textbf{x}$-SM and
$3$-$\textbf{x}$-LCX sequence with $\textbf{x}=(m,r,q)$. Thus by
Theorem \ref{thm+Ring+PSE+SM} (i),  the corresponding row-generating
function $T_n(a_0,a_2,b_0,b_2,q)$ in (\ref{eq+fst}) can be
considered as a five-variable refinement of $w_{n}(m,r,q)$.
\end{ex}

\begin{ex}[\textbf{Stirling permutations}]
Stirling permutations were introduced by Gessel and Stanley
\cite{GS78}. A Stirling permutation of order $n$ is a permutation of
the multiset $\{1^2,2^2,3^2, \ldots, n^2\}$ such that every element
between the two occurrences of $i$ are greater than $i$ for each
$i\in [n]$, where $[n] = \{1, 2, \ldots, n\}$. Denote by $Q_n$ the
set of Stirling permutations of order $n$. For $\sigma
=\sigma_1\sigma_2\ldots\sigma_{2n}\in Q_n$, an index $i\in[2n-1]$ is
an ascent plateau if $\sigma_{i-1} < \sigma_i = \sigma_{i+1}$. Let
$ap(\sigma)$ be the number of the ascent plateaus of $\sigma$. Let
$N_{n,k}=|\{\sigma\in Q_n: ap(\sigma)=k\}$. It is known that the
array $[N_{n,k}]_{n,k\geq0}$ satisfies the recurrence relation
\begin{eqnarray}
N_{n+1,k}&=&2kN_{n,k}+(2n-2k+3)N_{n,k-1}
\end{eqnarray}
with initial conditions $N_{1,1}=1$ and $N_{1,k}=0$ for $k\geq2$ or
$k\leq0$, see \cite{MY15}. Additionally, it is also related to the
perfect matching. A perfect matching of $[2n]$ is a set partition of
$[2n]$ with blocks (disjoint nonempty subsets) of size exactly $2$.
Let $M_{2n}$ be the set of matchings of $[2n]$, and let $M\in
M_{2n}$. Then $N_{n,k}$ is also the number of perfect matchings in
$M_{2n}$ with the restriction that only $k$ matching pairs have odd
smaller entries, see \cite{MY15}. For brevity, we let $N_{0,0}=1$
and $N_{0,k}=0$ for $k>1$ or $k<0$. It is clear that
$[N_{n,k}]_{n,k\geq0}$ satisfies the recurrence relation
\begin{eqnarray}
N_{n,k}&=&2kN_{n-1,k}+(2n-2k+1)N_{n-1,k-1}
\end{eqnarray}
with $N_{0,0}=1$ and $N_{0,k}=0$ for $k\geq1$ or $k<0$. Let
$N_n(q)=\sum_{k\geq0}N_{n,k}q^k$ for $n\in \mathbb{N}$. By Theorem
\ref{thm+Ring+PSE+SM} (v) and (\ref{CF+EF}), we have the continued
fraction expansion
\begin{eqnarray*}
\sum_{n\geq0}N_n(q)z^n&=&\s[(2n+1)q,2(n+1);z]_{n\geq0}\\
&=&\J[(2n+1)q+2n,2(n+1)(2n+1)q;z]_{n\geq0}
\end{eqnarray*}
and $(N_n(q))_{n\geq0}$ is $q$-SM and $3$-$q$-log-convex.
\end{ex}

\begin{ex}[\textbf{Minimax trees}]
Let $\mathcal {M}_{n,k}$ denote the set of all trees with $n$
vertices and $k$ leaves. Denote by $hr(T)$ the number of inner
vertices of the second kind of a minimax tree $T$. Denote by $cr(T)$
the number of inner vertices $s$ of a minimax tree $T$ which have
maximum label (it means that the label is maximum in the subtree
$T(s)$ of root $s$) so that an increasing tree $T$ is a tree
satisfying $cr(T)=0$. Let
$\overrightarrow{m}_{n,k}(p,q)=m_{n+1,k+1}(p,q)$, where
$$m_{n,k}(p,q):=\sum_{T\in \mathcal {M}_{n,k}}p^{hr(T)}q^{cr(T)}$$ for $n,k\geq0$ and $\overrightarrow{m}_{0,0}(p,q)=1$.
Then \begin{eqnarray}\label{ex+FH}
\overrightarrow{m}_{n,k}(p,q)=(1+p)(q+1)(k+1)\overrightarrow{m}_{n-1,k}(p,q)+(n-2k+1)\overrightarrow{m}_{n-1,k-1}(p,q)
\end{eqnarray}
for $0\leq k\leq (n+1)/2$, see Foata and Han \cite[Propositon
3.1]{FH01}. Let its row-generating function
$\mathscr{M}_n(x)=\sum_{k\geq0}\overrightarrow{m}_{n,k}(p,q)x^k$.
Then by (\ref{CF++++}), we have
\begin{eqnarray}
\sum\limits_{n=0}^{\infty}\mathscr{M}_n(x)
z^n=\J[(1+p)(q+1)(n+1),\frac{(1+p)(q+1)(n+2)(n+1)x}{2};z]_{n\geq0}.
\end{eqnarray}
It is easy to check that $(\mathscr{M}_n(x))_{n\geq0}$ is not
$x$-log-convex. But its reciprocal polynomial $\mathscr{M}^*_n(x)$
has the following property by the proof of (vi) in Theorem
\ref{thm+Ring+PSE+SM}.

\begin{prop}
For $n\geq0$, $\mathscr{M}^*_n\left(x+\frac{2}{(p+1)(q+1)}\right)$
form an $\textbf{x}$-SM and $3$-$\textbf{x}$-LCX sequence with
$\textbf{x}=(x,p,q)$.
\end{prop}
\end{ex}

\begin{ex}[\textbf{Peak Statistics}]
Let $pk(\pi)$ and $lpk(\pi)$ denote the numbers of interior peaks
and the numbers of left peaks of $\pi\in S_n$, respectively. Let
\begin{eqnarray*}
W_{n}(q)&=&\sum_{\pi\in S_n}q^{pk(\pi)}=\sum_{k\geq0}W_{n,k}q^{k},\\
\widetilde{W}_{n}(q)&=&\sum_{\pi\in
S_n}q^{lpk(\pi)}=\sum_{k\geq0}\widetilde{W}_{n,k}q^{k}.
\end{eqnarray*}
It is known that
\begin{eqnarray*}
W_{n,k}&=&(2k+2)W_{n-1,k}+(n-2k)W_{n-1,k-1},\\
\widetilde{W}_{n,k}&=&(2k+1)\widetilde{W}_{n-1,k}+(n-2k+1)\widetilde{W}_{n-1,k-1},
\end{eqnarray*}
where  $W_{1,0}=1$ and $\widetilde{W}_{0,0}=1$, see Stembridge
\cite{St97}, Petersen \cite{Pet07} and \cite[A008303, A008971]{Slo}
for instance. Then by (\ref{CF++++}), we have
\begin{eqnarray}
\sum\limits_{n=0}^{\infty}W_{n+1}(q)
t^n&=&\J[2(n+1),(n+2)(n+1)q;t]_{n\geq0},\\
\sum\limits_{n=0}^{\infty}\widetilde{W}_{n}(q)
t^n&=&\J[(2n+1),(n+1)^2q;t]_{n\geq0}.
\end{eqnarray}
It is easy to check that their $q$-log-convexity does not hold.
Thus, both $(W_{n+1}(q))_{n\geq0}$ and
$(\widetilde{W}_{n}(q))_{n\geq0}$ are not $q$-SM sequences. But
their reciprocal polynomials $W^*_{n+1}(q)$ and
$\widetilde{W}^*_{n}(q)$ have the following property by (vi) in
Theorem \ref{thm+Ring+PSE+SM}.
\begin{prop}
Both $(W^*_{n+1}(q+1))_{n\geq0}$ and
$(\widetilde{W}^*_{n}(q+1))_{n\geq0}$ are $q$-SM and $3$-$q$-LCX
sequences.
\end{prop}
\end{ex}

\begin{rem}
If we change the $\mathbb{R}[x]$ to an ordered commutative ring,
then we can extend results of this paper to those in the commutative
ring, see \cite{Zhu2018}.
\end{rem}
\section{Acknowledgments}
The author is extremely grateful to the anonymous referee for
his/her many valuable remarks and suggestions to improve the
original manuscript.


\end{document}